\documentclass[3p,11pt]{elsarticle}
\usepackage{amssymb}
\usepackage{amsmath}
\usepackage{amsthm}

\DeclareMathAlphabet{\mathcal}{OMS}{cmsy}{m}{n}

\makeatletter
\def\ps@pprintTitle{%
 \let\@oddhead\@empty
 \let\@evenhead\@empty
 \def\@oddfoot{\centerline{\thepage}}%
 \let\@evenfoot\@oddfoot}
\makeatother

\newcommand{\bbC}{\mathbb{C}}
\newcommand{\bbF}{\mathbb{F}}

\newcommand{\bbR}{\mathbb{R}}
\newcommand{\bbZ}{\mathbb{Z}}

\newcommand{\bfA}{\mathbf{A}}

\newcommand{\bfB}{\mathbf{B}}

\newcommand{\bfE}{\mathbf{E}}
\newcommand{\bfF}{\mathbf{F}}
\newcommand{\bfg}{\mathbf{g}}
\newcommand{\bfG}{\mathbf{G}}
\newcommand{\bfH}{\mathbf{H}}
\newcommand{\bfI}{\mathbf{I}}
\newcommand{\bfJ}{\mathbf{J}}

\newcommand{\bfX}{\mathbf{X}}

\newcommand{\bfz}{\mathbf{z}}
\newcommand{\bfZ}{\mathbf{Z}}

\newcommand{\bff}{\mathbf{f}}

\newcommand{\bfone}{\boldsymbol{1}}
\newcommand{\bfzero}{\boldsymbol{0}}

\newcommand{\bfphi}{\boldsymbol{\varphi}}
\newcommand{\bfPi}{\boldsymbol{\Pi}}
\newcommand{\bfpsi}{\boldsymbol{\psi}}
\newcommand{\bfPhi}{\boldsymbol{\Phi}}
\newcommand{\bfPsi}{\boldsymbol{\Psi}}

\newcommand{\calD}{\mathcal{D}}
\newcommand{\calG}{\mathcal{G}}

\newcommand{\calV}{\mathcal{V}}

\newcommand{\rmi}{\mathrm{i}}
\newcommand{\rmT}{\mathrm{T}}

\newcommand{\Tr}{\operatorname{Tr}}
\newcommand{\QSD}{{\operatorname{QSD}}}
\newcommand{\SRG}{{\operatorname{SRG}}}
\newcommand{\BIBD}{{\operatorname{BIBD}}}
\newcommand{\RBIBD}{{\operatorname{RBIBD}}}

\newcommand{\Fro}{\mathrm{Fro}}

\newcommand{\abs}[1]{|{#1}|}

\newcommand{\bigbracket}[1]{\bigl[{#1}\bigr]}
\newcommand{\Bigbracket}[1]{\Bigl[{#1}\Bigr]}
\newcommand{\biggbracket}[1]{\biggl[{#1}\biggr]}

\newcommand{\set}[1]{\{{#1}\}}

\newcommand{\norm}[1]{\|{#1}\|}

\newcommand{\ip}[2]{\langle{#1},{#2}\rangle}

\newtheorem{theorem}{Theorem}

\newtheorem{proposition}{Proposition}
\newtheorem{corollary}{Corollary}

\theoremstyle{definition}
\newtheorem{definition}{Definition}

\newtheorem{conjecture}{Conjecture}

\begin{document}
\begin{frontmatter}
\title{Hadamard Equiangular Tight Frames}

\author[AFIT]{Matthew Fickus}
\ead{Matthew.Fickus@gmail.com}
\author[Cincy]{John Jasper}
\author[AFIT]{Dustin G.\ Mixon}
\author[AFIT]{Jesse D.\ Peterson}

\address[AFIT]{Department of Mathematics and Statistics, Air Force Institute of Technology, Wright-Patterson AFB, OH 45433}
\address[Cincy]{Department of Mathematical Sciences, University of Cincinnati, Cincinnati, OH 45221}


\begin{abstract}
An equiangular tight frame (ETF) is a type of optimal packing of lines in Euclidean space.
They are often represented as the columns of a short, fat matrix.
In certain applications we want this matrix to be flat, that is, have the property that all of its entries have modulus one.
In particular, real flat ETFs are equivalent to self-complementary binary codes that achieve the Grey-Rankin bound.
Some flat ETFs are (complex) Hadamard ETFs, meaning they arise by extracting rows from a (complex) Hadamard matrix.
These include harmonic ETFs, which are obtained by extracting the rows of a character table that correspond to a difference set in the underlying finite abelian group.
In this paper, we give some new results about flat ETFs.
One of these results gives an explicit Naimark complement for all Steiner ETFs,
which in turn implies that all Kirkman ETFs are possibly-complex Hadamard ETFs.
This in particular produces a new infinite family of real flat ETFs.
Another result establishes an equivalence between real flat ETFs and certain types of quasi-symmetric designs,
resulting in a new infinite family of such designs.
\end{abstract}

\end{frontmatter}

\section{Introduction}

An equiangular tight frame is a type of optimal packing of lines in Euclidean space.
To be precise, if $\set{\bfphi_j}_{j=1}^{n}$ is any sequence of nonzero, equal-norm vectors in $\bbF^d$ where $n\geq d$ and $\bbF$ is either $\bbR$ or $\bbC$, the \textit{coherence} of $\set{\bfphi_j}_{j=1}^{n}$ is bounded below by the \textit{Welch bound}~\cite{Welch74}:
\begin{equation}
\label{equtaion.Welch bound}
\max_{j\neq j'}\frac{\abs{\ip{\bfphi_j}{\bfphi_{j'}}}}{\norm{\bfphi_j}\norm{\bfphi_{j'}}}
\geq\biggbracket{\frac{n-d}{d(n-1)}}^{\frac12}.
\end{equation}
In the case where $\bbF=\bbR$,
each vector $\bfphi_j$ spans a real line,
and the coherence is the cosine of the smallest interior angle between any pair of these lines.
In this case, if equality in~\eqref{equtaion.Welch bound} is achieved then this smallest pairwise angle is as large as possible, meaning the lines are optimally packed.

For any vectors \smash{$\set{\bfphi_j}_{j=1}^{n}$} in $\bbF^d$,
the corresponding \textit{synthesis operator} $\bfPhi$ is the $d\times n$ matrix whose $j$th column is \smash{$\bfphi_j$}.
It is well known~\cite{StrohmerH03} that a sequence of nonzero equal-norm vectors \smash{$\set{\bfphi_j}_{j=1}^{n}$} in $\bbF^d$ achieves equality in~\eqref{equtaion.Welch bound} if and only if they form an \textit{equiangular tight frame} (ETF) for $\bbF^d$, that is, if and only if the rows of its synthesis operator $\bfPhi$ are equal-norm and orthogonal (tightness) while $\abs{\ip{\bfphi_j}{\bfphi_{j'}}}$ is constant over all $j\neq j'$ (equiangularity).
These conditions are restrictive, and ETFs are not easy to find.
That said, a growing number of explicit constructions of them are known~\cite{FickusM16}.
Real ETFs in particular are known to be equivalent to a special class of \textit{strongly regular graphs} (SRGs)~\cite{StrohmerH03,HolmesP04,Waldron09}.
Beyond these, the most popular ETFs are \textit{harmonic ETFs}~\cite{StrohmerH03,XiaZG05,DingF07}, which are obtained by restricting the characters of a finite abelian group $\calG$ to a \textit{difference set}, namely a subset $\calD$ of $\calG$ with the property that the cardinality of $\set{(i,i')\in\calD\times\calD: g=i-i'}$ is constant over all nonzero $g\in\calG$.
For example, $\set{0001,0101,0010,1010,0011,1111}$ is a difference set in $\bbZ_2^4$,
and the corresponding six rows of the canonical Hadamard matrix of size $16$ yields the following matrix whose columns give an optimal packing of $16$ lines in $\bbR^6$:
\begin{equation}
\label{equation.6x16 flat}
\bfPhi
=\left[\begin{array}{rrrrrrrrrrrrrrrr}
 1&-1& 1&-1& 1&-1& 1&-1& 1&-1& 1&-1& 1&-1& 1&-1\\
 1& 1&-1&-1& 1& 1&-1&-1& 1& 1&-1&-1& 1& 1&-1&-1\\
 1&-1&-1& 1& 1&-1&-1& 1& 1&-1&-1& 1& 1&-1&-1& 1\\
 1&-1& 1&-1&-1& 1&-1& 1& 1&-1& 1&-1&-1& 1&-1& 1\\
 1& 1&-1&-1& 1& 1&-1&-1&-1&-1& 1& 1&-1&-1& 1& 1\\
 1&-1&-1& 1&-1& 1& 1&-1&-1& 1& 1&-1& 1&-1&-1& 1
\end{array}\right].
\end{equation}
Every harmonic ETF is \textit{unital}, meaning each entry of its synthesis operator $\bfPhi$ is a root of unity;
number-theoretic necessary conditions on the existence of unital ETFs are given in~\cite{SustikTDH07}.
Every harmonic ETF also has the following two special properties:

\begin{definition}
\label{definition.flat}
An ETF $\set{\bfphi_j}_{j=1}^{n}$ for $\bbF^d$ is \textit{(complex) Hadamard} if its $d\times n$ synthesis operator $\bfPhi$ is a submatrix of an $n\times n$ (complex) Hadamard matrix,
and is \textit{flat} if each entry of $\bfPhi$ is unit modulus.
\end{definition}

In this paper, we study Hadamard ETFs and flat ETFs in general.
Flat ETFs in particular arise in several applications, as detailed below.
Because of this, we would like to have many ways to construct them.
However, to date we only have two ways: besides harmonic ETFs, the only known flat ETFs are \textit{Kirkman ETFs}~\cite{JasperMF14}, which are a special class of \textit{Steiner ETFs}~\cite{FickusMT12}.
In this paper,
we give some new results about flat ETFs which we hope will better inform future searches for them.
In particular, we show that every Kirkman ETF is a possibly-complex Hadamard ETF
(Theorem~\ref{theorem.Naimark complement of Kirkman ETF}) and give a new characterization of all real flat ETFs (Theorem~\ref{theorem.QSD}).

Flat ETFs are especially attractive for certain applications involving coding theory and waveform design.
In particular, real flat ETFs are equivalent to self-complementary binary codes whose minimum-pairwise-Hamming distance is as large as possible, achieving the \textit{Grey-Rankin bound}~\cite{JasperMF14}.
More generally, ETFs have been proposed as waveforms for wireless communication~\cite{StrohmerH03}.
In that setting, flat ETFs allow the transmitted signals to have maximal energy ($2$-norm) subject to real-world bounds on transmitter power ($\infty$-norm), while still interfering with each other as little as possible.
This same rationale is part of the reason why Hadamard matrices are used in traditional CDMA,
and why \textit{constant amplitude zero-autocorrelation} waveforms have been proposed as radar waveforms~\cite{BenedettoD07}.
It also helped motivate the investigation into unital ETFs given in~\cite{SustikTDH07};
a real ETF is flat if and only if it is unital.
Recently, certain complex flat ETFs also have been used to construct tight frames for $\bbC^d$ that  consist of $d^2+1$ vectors and have minimal coherence, meeting the \textit{orthoplex bound}~\cite{BodmannH16}.
These frames present a reasonable alternative to ETFs for $\bbC^d$ that consist of $d^2$ vectors;
though such ETFs have been much sought after in quantum information theory, their existence remains unsettled in all but a finite number of cases~\cite{RenesBSC04,Zauner99}.
Real flat sensing matrices with low coherence also arise in certain compressed sensing applications like the Single Pixel Camera~\cite{DuarteDTLSKB08},
though so far, all known real flat ETFs with $n-1>d>1$ are known to be inferior to random $\set{\pm1}$-valued matrices with respect to the \textit{restricted isometry property}~\cite{BandeiraFMW13,JasperMF14}.

In the next section, we establish notation and discuss some well-known concepts that we need, such as \textit{Naimark complements} of ETFs.
In Section~3,
we construct an explicit Naimark complement of any Steiner ETF,
and use it to show that every Kirkman ETF is a possibly-complex Hadamard ETF.
This in particular yields a new infinite family of real flat ETFs.
In Section~4,
we characterize all real flat ETFs in terms of combinatorial designs known as \textit{quasi-symmetric designs} (QSDs).
When combined with the results of Section~3, this characterization provides a new infinite family of QSDs.
When combined with~\cite{BrackenMW06}, this characterization shows that there exists a real flat ETF when $(d,n)$ is either $(66,144)$ or $(78,144)$.
This characterization also leads to new necessary conditions on the existence of real flat ETFs,
showing for example that real ETFs with $(d,n)=(15,36)$ are not flat.
In the fifth and final section, we give some other miscellaneous results about Hadamard ETFs.

One may also consider flat representations of higher dimension,
such as when the columns of an $m\times n$ flat matrix $\bfPhi$ form an ETF for their $d$-dimensional span where $d<m$.
Such ETFs have arisen only recently~\cite{FickusMJ16,FickusJMP18}, and we leave a more thorough investigation of them for future work.
Preliminary versions of parts of the material presented in Sections~3 and~4 appeared in the conference proceedings~\cite{FickusJMP15a} and~\cite{FickusJMP15b}, respectively.

\section{Preliminaries}

As above, let $\bfPhi$ denote the $d\times n$ synthesis operator of \smash{$\set{\bfphi_j}_{j=1}^{n}$},
and let $\bfPhi^*$ be its $n\times d$ conjugate transpose.
The corresponding \textit{Gram matrix} is the $n\times n$ matrix $\bfPhi^*\bfPhi$ whose $(j,j')$th entry is $(\bfPhi^*\bfPhi)(j,j')=\ip{\bfphi_j}{\bfphi_{j'}}$.
We say \smash{$\set{\bfphi_j}_{j=1}^{n}$} is \textit{equal-norm} if there exists $\beta>0$ such that $\norm{\bfphi_j}^2=\beta$ for all $j$,
and say it is \textit{equiangular} if we further have $\gamma\geq0$ such that $\abs{\ip{\bfphi_j}{\bfphi_{j'}}}=\gamma$ for all $j\neq j'$.
That is, \smash{$\set{\bfphi_j}_{j=1}^{n}$} is equiangular when the diagonal entries of $\bfPhi^*\bfPhi$ are constant and its off-diagonal entries have constant modulus.
We say \smash{$\set{\bfphi_j}_{j=1}^{n}$} is a \textit{tight frame} for $\bbF^d$ if there exists $\alpha>0$ such that $\bfPhi\bfPhi^*=\alpha\bfI$,
namely when the rows of $\bfPhi$ are nonzero, equal-norm and orthogonal.
As mentioned above, it is well known that a sequence of equal-norm vectors \smash{$\set{\bfphi_j}_{j=1}^{n}$} achieves equality in~\eqref{equtaion.Welch bound} if and only if it is both equiangular and a tight frame for $\bbF^d$, namely when it is an ETF for $\bbF^d$;
see~\cite{JasperMF14} for a short, modern proof of this fact.
In this case, the tightness constant $\alpha$ is necessarily $\frac{n\beta}{d}$,
and the coherence $\frac{\gamma}{\beta}$ is necessarily the Welch bound \smash{$[\tfrac{n-d}{d(n-1)}]^{\frac12}$}.

For any ETF $\set{\bfphi_j}_{j=1}^{n}$ for $\bbF^d$ with $n>d$,
there exists an ETF \smash{$\set{\widetilde{\bfphi}_j}_{j=1}^{n}$} for $\bbF^{n-d}$.
Here, the $n-d$ rows of the corresponding synthesis operator \smash{$\widetilde{\bfPhi}$} are formed by completing the $d$ rows of $\bfPhi$ to an equal-norm orthogonal basis for $\bbF^n$.
This ensures that
\begin{equation*}
\alpha^{-\frac12}\left[\begin{array}{cc}\bfPhi\\\widetilde{\bfPhi}\end{array}\right]
\end{equation*}
is unitary, implying that $\widetilde{\bfPhi}\widetilde{\bfPhi}^*=\alpha\bfI$ and that $\widetilde{\bfPhi}^*\widetilde{\bfPhi}=\alpha\bfI-\bfPhi^*\bfPhi$,
meaning that the columns $\set{\widetilde{\bfphi}_j}_{j=1}^{n}$ of $\widetilde{\bfPhi}$ form a tight frame for $\bbF^{n-d}$ with the property that $\norm{\widetilde{\bfphi}_j}^2=\alpha-\beta$ for all $j$ and that
\smash{$\abs{\ip{\widetilde{\bfphi}_j}{\widetilde{\bfphi}_{j'}}}=\abs{-\ip{\bfphi_j}{\bfphi_{j'}}}=\gamma$} for all \smash{$j\neq j'$}.
Any such sequence \smash{$\set{\widetilde{\bfphi}_j}_{j=1}^{n}$} is called a \textit{Naimark complement} for $\bbF^{n-d}$.
Since the rows of $\widetilde{\bfPhi}$ can be any appropriately-scaled orthogonal basis for the orthogonal complement of the row space of $\bfPhi$,
Naimark complements are not unique.
This is a key point in distinguishing Hadamard ETFs from those that are simply flat,
as expressed in the following restatement of Definition~\ref{definition.flat},
whose proof is immediate:

\begin{proposition}
\label{proposition.Hadamard ETF characterization}
A real flat ETF is Hadamard if and only if it has a real flat Naimark complement.
A complex flat ETF is complex Hadamard if and only if it has a flat Naimark complement.
\end{proposition}

For a simple example of these ideas, consider \textit{regular simplices},
namely ETFs for $\bbF^d$ that consist of $d+1$ vectors.
The Naimark complement of a regular simplex is an ETF for $\bbF^1$,
namely a sequence of $d+1$ nonzero scalars of equal modulus.
Conversely, any such sequence is an ETF for $\bbF^1$,
meaning each of its Naimark complements is a regular simplex.
In particular, for any positive integer $d$, there exists a regular simplex for $\bbF^d$.
When $\bbF=\bbC$, this ETF can always be chosen to be flat:
we can form its synthesis operator by removing any row from a complex Hadamard matrix of size $d+1$,
such as a discrete Fourier transform (DFT).
This same method produces a flat regular simplex for $\bbR^d$ if there exists a Hadamard matrix of size $d+1$.
For example, removing the all-ones row from the canonical $4\times 4$ Hadamard matrix gives a flat regular simplex for $\bbR^3$ (a tetrahedron):
\begin{equation}
\label{equation.3x4 real flat regular simplex}
\left[\begin{array}{rrrr}
 1&-1& 1&-1\\
 1& 1&-1&-1\\
 1&-1&-1& 1\\
\end{array}\right].
\end{equation}
In fact, since the Naimark complement of a regular simplex is automatically flat,
this is the only way to produce a real flat regular simplex,
meaning each such ETF is Hadamard.
In particular, by the known half of the Hadamard conjecture,
a real flat regular simplex for $\bbR^d$ can only exist if $d+1$ is either $2$ or divisible by $4$.

For another example of these ideas, note that every harmonic ETF is possibly-complex Hadamard,
being a $d\times n$ submatrix of an $n\times n$ character table of some finite abelian group.
In this case, the remaining $n-d$ rows of the character table give a flat Naimark complement.
Some of these ETFs are real,
such as those that arise from a \textit{McFarland} difference set in \smash{$\bbZ_2^{2(e+1)}$} where $e$ is a positive integer~\cite{DingF07,JasperMF14}.
These ETFs have parameters \smash{$d=\frac12(n-\sqrt{n})$} where $n=2^{2(e+1)}$;
taking $e=1$ gives the real flat ETF given in~\eqref{equation.6x16 flat}.
By Corollary~2 of~\cite{JasperMF14}, every real harmonic ETF is either a regular simplex or has these parameters.

Thus, the two most popular constructions of flat ETFs---regular simplices and harmonic ETFs---also happen to be possibly-complex Hadamard.
This leads one to ask whether this is true in general:

\begin{conjecture}
Every real flat ETF is Hadamard.
\end{conjecture}

\begin{conjecture}
Every complex flat ETF is complex Hadamard.
\end{conjecture}

To be clear, it is reasonable to believe these conjectures are false,
since apart from the above examples, there is no compelling reason to believe that every flat ETF necessarily has a flat Naimark complement.
That said, we have not been able to prove that either one of these conjectures is false.
In fact, in the next section we show to the contrary that the other known class of flat ETFs---the Kirkman ETFs of~\cite{JasperMF14}---are possibly-complex Hadamard.
Both the proof of that result as well as the characterization of real flat ETFs given in Section~4 depend on certain well-known combinatorial designs.

To elaborate, given integers $v$, $k$, $\lambda$, $r$ and $b$ with $v>k>0$ and $b>0$,
a corresponding \textit{balanced incomplete block design} $\BIBD(v,k,\lambda,r,b)$ is a set $\calV$ of cardinality $v$---whose elements are called \textit{vertices}---along with $b$ subsets of $\calV$---called \textit{blocks}---with the property that each block contains exactly $k$ vertices,
each vertex is contained in exactly $r$ blocks,
and every pair of distinct vertices is contained in exactly $\lambda$ blocks.
Letting $\bfX$ be a corresponding $b\times v$ incidence matrix,
this is equivalent to having
\begin{equation}
\label{equation.BIBD properties}
\bfX\bfone=k\bfone,
\quad
\bfX^\rmT\bfone=r\bfone,
\quad
\bfX^\rmT\bfX=(r-\lambda)\bfI+\lambda\bfJ,
\end{equation}
where $\bfone$ and $\bfJ$ denote all-ones vectors and matrices, respectively.
For example, taking all two-element subsets of $\calV=\set{1,2,3,4}$ gives a $\BIBD(4,2,1,3,6)$ with an incidence matrix of
\begin{equation}
\label{equation.6x4 incidence matrix}
\bfX=\left[\begin{array}{cccc}
1&1&0&0\\
0&0&1&1\\
1&0&1&0\\
0&1&0&1\\
1&0&0&1\\
0&1&1&0
\end{array}\right].
\end{equation}
The parameters of a BIBD are dependent,
satisfying
\begin{equation}
\label{equation.BIBD parameters relationship}
bk=vr,
\quad
(v-1)\lambda=r(k-1).
\end{equation}
Since $v>k>0$ and $b>0$, these relationships imply $0\leq\lambda<r<b$.
This in turn implies that $\bfX^\rmT\bfX=(r-\lambda)\bfI+\lambda\bfJ$ is positive-definite, implying \textit{Fisher's inequality}, namely that $b\geq v$.

\section{Naimark complements of Steiner and Kirkman equiangular tight frames.}

We begin this section by constructing an explicit Naimark complement for any Steiner ETF.
Every Steiner ETF arises from a $\BIBD(v,k,\lambda,r,b)$ with $\lambda=1$;
such a combinatorial design is also known as a \textit{Steiner system} $\operatorname{S}(t,k,v)$ with $t=2$.
To be precise,
let $\bfX$ be the $b\times v$ incidence matrix of a $\BIBD(v,k,1,r,b)$,
let $\bfPsi$ be the $r\times(r+1)$ synthesis operator of a flat regular simplex.
The synthesis operator $\bfPhi$ of the corresponding Steiner ETF is the $b\times v(r+1)$ matrix obtained by replacing each $1$-valued entry in any given column of $\bfX$ with a distinct row of $\bfPsi$, and replacing each $0$-valued entry of $\bfX$ with a $1\times(r+1)$ submatrix of zeros.
For example, combining~\eqref{equation.6x4 incidence matrix} and~\eqref{equation.3x4 real flat regular simplex} in this way gives the following Steiner ETF with parameters $(d,n)=(6,16)$:
\begin{equation}
\label{equation.6x16 Steiner ETF}
\bfPhi
=\left[\begin{array}{rrrrrrrrrrrrrrrr}
 1&-1& 1&-1&\phantom{-}1&-1& 1&-1&\phantom{-}0& 0& 0& 0&\phantom{-}0& 0& 0& 0\\
 0& 0& 0& 0& 0& 0& 0& 0& 1&-1& 1&-1& 1&-1& 1&-1\\
 1& 1&-1&-1& 0& 0& 0& 0& 1& 1&-1&-1& 0& 0& 0& 0\\
 0& 0& 0& 0& 1& 1&-1&-1& 0& 0& 0& 0& 1& 1&-1&-1\\
 1&-1&-1& 1& 0& 0& 0& 0& 0& 0& 0& 0& 1&-1&-1& 1\\
 0& 0& 0& 0& 1&-1&-1& 1& 1&-1&-1& 1& 0& 0& 0& 0
\end{array}\right].
\end{equation}
This construction was introduced in~\cite{GoethalsS70} as a method for obtaining SRGs.
In~\cite{FickusMT12} it was rediscovered and recognized as a method for constructing ETFs.
These ETFs are real if the flat regular simplex is real, namely when it is obtained from a Hadamard matrix of size $r+1$.
This method of combining two matrices is unlike other common methods for doing so, as it explicitly relies on the fact that $\bfX$ is $\set{0,1}$-valued and has constant column sums.
Because of the unusualness of this construction,
the standard proof of the fact that it yields ETFs is ``wordy," see~\cite{GoethalsS70,FickusMT12}.
This is frustrating, since $\bfPhi$ seems tantalisingly similar to the tensor product of $\bfX$ and $\bfPhi$.
In fact, as we now explain, an explicit tensor-product-based expression for $\bfPhi$ can be found, provided we first ``lift" the incidence matrix of a BIBD to a permutation matrix:

\begin{definition}
\label{definition.BIBD permutation matrix}
Let $\bfX$ be the $b\times v$ incidence matrix of a $\BIBD(v,k,1,r,b)$.
A corresponding \textit{BIBD permutation matrix} is a permutation matrix $\bfPi$ of size $bk=vr$ such that
\begin{equation}
\label{equation.definition of BIBD permutation matrix}
\bfX
=(\bfI_b^{}\otimes\bfone_k^\rmT)\bfPi(\bfI_v\otimes\bfone_r)^{}.
\end{equation}
\end{definition}

Here, $\bfI_b$ and $\bfI_v$ are $b\times b$ and $v\times v$ identity matrices, respectively,
while $\bfone_k$ and $\bfone_r$ are all-ones vectors of size $k\times1$ and $r\times 1$, respectively.
Regarding $\bfPi$ as a $b\times v$ array of $\set{0,1}$-valued submatrices $\set{\bfPi_{i,j}}_{i=1,}^{b}\,_{j=1}^{v}$ of size $k\times r$,
this means that for any $i=1,\dotsc,b$ and $j=1,\dotsc,v$,
\begin{equation}
\label{equation.definition of BIBD permutation matrix restated}
\bfX(i,j)
=\bfone_k^\rmT\bfPi_{i,j}\bfone_r^{}
=\sum_{p=1}^{k}\sum_{q=1}^{r}\bfPi_{i,j}(p,q).
\end{equation}
Thus, $\bfPi_{i,j}=\bfzero$ if $\bfX(i,j)=0$ and contains exactly one $1$-valued entry if $\bfX(i,j)=1$.
The permutation matrix of a BIBD is not unique, but one does always exist.
For example,
we can define $\bfPi_{i,j}(p,q)$ to be $1$ if and only if $\bfX(i,j)$ is both the $p$th ``$1$" in the $i$th row of $\bfX$ and the $q$th ``$1$" in the $j$th column of $\bfX$;
for the incidence matrix~\eqref{equation.6x4 incidence matrix}, this gives
\begin{equation}
\label{equation.12x12 BIBD permutation matrix}
\setlength{\arraycolsep}{4pt}
\bfPi
=\left[\begin{array}{rrr|rrr|rrr|rrr}
1&0&0&0&0&0&0&0&0&0&0&0\\
0&0&0&1&0&0&0&0&0&0&0&0\\\hline
0&0&0&0&0&0&1&0&0&0&0&0\\
0&0&0&0&0&0&0&0&0&1&0&0\\\hline
0&1&0&0&0&0&0&0&0&0&0&0\\
0&0&0&0&0&0&0&1&0&0&0&0\\\hline
0&0&0&0&1&0&0&0&0&0&0&0\\
0&0&0&0&0&0&0&0&0&0&1&0\\\hline
0&0&1&0&0&0&0&0&0&0&0&0\\
0&0&0&0&0&0&0&0&0&0&0&1\\\hline
0&0&0&0&0&1&0&0&0&0&0&0\\
0&0&0&0&0&0&0&0&1&0&0&0
\end{array}\right].
\end{equation}
This method produces a permutation matrix since the corresponding mappings from $[b]\times[k]$ and $[v]\times[r]$ into $\set{(i,j): \bfX(i,j)=1}$ are both invertible,
establishing a permutation $(i,p)\mapsto(j,q)$.

A Steiner ETF is obtained by taking a BIBD's permutation matrix $\bfPi$,
and replacing the vector $\bfone_r$
in~\eqref{equation.definition of BIBD permutation matrix} with an $r\times(r+1)$ unimodular regular simplex $\bfPsi$,
that is, by letting
\begin{equation}
\label{equation.Steiner ETF tensor definition}
\bfPhi=(\bfI_b^{}\otimes\bfone_k^\rmT)\bfPi(\bfI_v\otimes\bfPsi)^{}.
\end{equation}
Indeed, regarding $\bfI_b^{}\otimes\bfone_k^\rmT$, $\bfPi$ and $\bfI_v\otimes\bfPsi$ as $b\times b$, $b\times v$ and $v\times v$ arrays of submatrices of size $1\times k$, $k\times r$ and $r\times(r+1)$, respectively, gives that $\bfPhi$ is a $b\times v$ array of submatrices of size $1\times(r+1)$.
In particular, the $(i,j)$th submatrix of $\bfPhi$ is $\bfPhi_{i,j}=\bfone_k^\rmT\bfPi_{i,j}\bfPsi$;
when $\bfX(i,j)=0$, this is a row of zeros;
when $\bfX(i,j)=1$, this is the row of $\bfPsi$ corresponding to the column index of the nonzero entry of $\bfPi_{i,j}$.
As part of the next result,
we directly verify that~\eqref{equation.Steiner ETF tensor definition} indeed defines the synthesis operator of an ETF with parameters $(d,n)=(b,v(r+1))$.
In fact, we show this holds even if $\bfone_k$ is replaced with any $k\times 1$ vector with unimodular entries;
this turns out to be the key to constructing an explicit Naimark complement of $\bfPhi$.

Motivating by example, note that when $\bfPi$ and $\bfPsi$ are~\eqref{equation.12x12 BIBD permutation matrix} and~\eqref{equation.3x4 real flat regular simplex}, respectively, the matrix $(\bfI_6\otimes\left[\begin{array}{rr}1&1\end{array}\right])\bfPi(\bfI_4\otimes\bfPsi)^{}$ gives the ETF in~\eqref{equation.6x16 Steiner ETF}.
Meanwhile, replacing $\left[\begin{array}{rr}1&1\end{array}\right]$ with the orthogonal flat vector
$\left[\begin{array}{rr}1&-1\end{array}\right]$ yields the $6\times 16$ synthesis operator of a second ETF:
\begin{equation}
\label{equation.6x16 complementary Steiner ETF}
(\bfI_6\otimes\left[\begin{array}{rr}1&-1\end{array}\right])\bfPi(\bfI_4\otimes\bfPsi)^{}
=\left[\begin{array}{rrrrrrrrrrrrrrrr}
 1&-1& 1&-1&-1& 1&-1& 1& 0& 0&\phantom{-}0& 0& 0& 0& 0&\phantom{-}0\\
 0& 0& 0& 0& 0& 0& 0& 0& 1&-1& 1&-1&-1& 1&-1& 1\\
 1& 1&-1&-1& 0& 0& 0& 0&-1&-1& 1& 1& 0& 0& 0& 0\\
 0& 0& 0& 0& 1& 1&-1&-1& 0& 0& 0& 0&-1&-1& 1& 1\\
 1&-1&-1& 1& 0& 0& 0& 0& 0& 0& 0& 0&-1& 1& 1&-1\\
 0& 0& 0& 0& 1&-1&-1& 1&-1& 1& 1&-1& 0& 0& 0& 0
\end{array}\right].
\end{equation}
In essence, this second ETF is obtained by negating the second nonzero $1\times(r+1)$ submatrix in each row of~\eqref{equation.6x16 Steiner ETF}.
As explicitly verified in the next result, this ensures that the row spaces of the synthesis operators of these two ETFs are mutually orthogonal.
In particular,
\eqref{equation.6x16 complementary Steiner ETF} gives six rows of a $10\times 16$ Naimark complement of~\eqref{equation.6x16 Steiner ETF}.
The remaining four rows can be obtained by tensoring $\bfI_v=\bfI_4$ with the $1\times(r+1)=1\times4$ Naimark complement of $\bfPsi$, and scaling it appropriately:
\begin{equation*}
\sqrt{k}(\bfI_v\otimes\bfone_{r+1}^*)
=\sqrt{2}(\bfI_4\otimes\bfone_4^*)
=\sqrt{2}\left[\begin{array}{rrrrrrrrrrrrrrrr}
1&1&1&1&0&0&0&0&0&0&0&0&0&0&0&0\\
0&0&0&0&1&1&1&1&0&0&0&0&0&0&0&0\\
0&0&0&0&0&0&0&0&1&1&1&1&0&0&0&0\\
0&0&0&0&0&0&0&0&0&0&0&0&1&1&1&1
\end{array}\right].
\end{equation*}
In the following result, we generalize this example to construct an explicit Naimark complement for any Steiner ETF.

\begin{theorem}
\label{theorem.Naimark complement of Steiner ETF}
Let $\bfPi$ be the permutation matrix of a $\BIBD(v,k,1,r,b)$---see Definition~\ref{definition.BIBD permutation matrix}---and let $\bfF$ and $\bfG$ be possibly-complex Hadamard matrices of size $k$ and $r+1$, respectively.
Write $\bfG=\left[\begin{array}{cc}\bfG_1&\bfg_2\end{array}\right]$ where $\bfG_1$ is $(r+1)\times r$.
For any $l=1,\dotsc,k$, let
\begin{equation}
\label{equation.tensor definition of Steiner ETF}
\bfPhi_{\smash{l}}
:=(\bfI_b^{}\otimes\bff_{\smash{l}}^*)\bfPi(\bfI_v\otimes\bfG_1^*),
\end{equation}
where $\bff_{\smash{l}}$ is the $l$th column of $\bfF$.
Then, the $v(r+1)$ columns of each $\bfPhi_{\smash{l}}$ form an ETF for $\bbF^b$.
Moreover, the row spaces of the matrices $\set{\bfPhi_{\smash{l}}}_{l=1}^{k}$ are mutually orthogonal and the columns of
\begin{equation}
\label{equation.Naimark complement of Steiner ETF}
\widetilde{\bfPhi}_1=\left[\begin{array}{c}
\bfPhi_2\\
\vdots\\
\bfPhi_k\\
\sqrt{k}(\bfI_v\otimes\bfg_2^*)
\end{array}\right]
\end{equation}
form a Naimark complement for the ETF formed by the columns of $\bfPhi_1$.
Here, both $\bfF$ and $\bfG$ can be chosen to be real if and only if $k=2$ and there exists a Hadamard matrix of size $v$.
In this case, the columns of $\bfPhi_1$ and $\smash{\widetilde{\bfPhi}_1}$ form real ETFs of $v^2$ vectors for $\bbR^{\frac12v(v-1)}$ and $\bbR^{\frac12v(v+1)}$, respectively.
\end{theorem}

\begin{proof}
For any $l,l'=1,\dotsc,k$, the fact that $\bfG_1^*\bfG_1^{}=(r+1)\bfI_r$ gives
\begin{align}
\nonumber
\bfPhi_{\smash{l}}^{}\bfPhi_{\smash{l'}}^*
&=[(\bfI_b^{}\otimes\bff_{\smash{l}}^*)\bfPi(\bfI_v\otimes\bfG_1^*)][(\bfI_v\otimes\bfG_1^{})\bfPi^\rmT(\bfI_b^{}\otimes\bff_{\smash{l'}})]\\
\nonumber
&=(\bfI_b^{}\otimes\bff_{\smash{l}}^*)\bfPi(\bfI_v\otimes\bfG_1^*\bfG_1^{})\bfPi^\rmT(\bfI_b^{}\otimes\bff_{\smash{l'}}^{})\\
\nonumber
&=(\bfI_b^{}\otimes\bff_{\smash{l}}^*)\bfPi[\bfI_v\otimes(r+1)\bfI_r]\bfPi^\rmT(\bfI_b^{}\otimes\bff_{\smash{l'}}^{})\\
\nonumber
&=(r+1)(\bfI_b^{}\otimes\bff_{\smash{l}}^*)(\bfI_b^{}\otimes\bff_{\smash{l'}}^{})\\
\nonumber
&=(r+1)(\bfI_b^{}\otimes\ip{\bff_{\smash{l}}}{\bff_{\smash{l'}}})\\
\label{equation.Naimark complement of Steiner ETF 0}
&=\left\{\begin{array}{cl}k(r+1)\bfI_b,&l=l',\\\bfzero_b,&l\neq l'.\end{array}\right.
\end{align}
Thus, the columns of each $\bfPhi_{\smash{l}}$ form a tight frame for $\bbF^b$,
and the row spaces of $\bfPhi_{\smash{l}}$ and $\bfPhi_{\smash{l'}}$ are orthogonal for any $l\neq l'$, as claimed.

To see that $\widetilde{\bfPhi}_1$ is a Naimark complement of $\bfPhi_1$, note that $\bfG_1^*\bfg_2^{}=\bfzero$ and so for any $l=1,\dotsc,k$,
\begin{equation*}
\bfPhi_{\smash{l}}[\sqrt{k}(\bfI_v\otimes\bfg_2^*)]^*
=\sqrt{k}[(\bfI_b^{}\otimes\bff_{\smash{l}}^*)\bfPi(\bfI_v\otimes\bfG_1^*)](\bfI_v\otimes\bfg_2)
=\sqrt{k}(\bfI_b^{}\otimes\bff_{\smash{l}}^*)\bfPi(\bfI_v\otimes\bfG_1^*\bfg_2^{})
=\bfzero.
\end{equation*}
Also,
$[\sqrt{k}(\bfI_v\otimes\bfg_2^*)][\sqrt{k}(\bfI_v\otimes\bfg_2^*)]^*
=k(\bfI_v\otimes\ip{\bfg_2}{\bfg_2})
=k(r+1)\bfI_v$.
When combined with~\eqref{equation.Naimark complement of Steiner ETF 0},
these facts imply that the rows of $\widetilde{\bfPhi}_1$ are orthogonal to each other as well as to all rows of $\bfPhi_1$, and that all rows of $\bfPhi_1$ and $\widetilde{\bfPhi}_1$ have the same norm.
Since $\bfPhi_1$ is $b\times v(r+1)$ while $\widetilde{\bfPhi}_1$ is $[b(k-1)+v]\times v(r+1)$ where~\eqref{equation.BIBD parameters relationship} gives
$b+[b(k-1)+v]=bk+v=vr+v=v(r+1)$,
this implies that the columns of $\widetilde{\bfPhi}_1$ indeed form a Naimark complement of the tight frame formed by the columns of $\bfPhi_1$.

Next, for any $l=1,\dotsc,k$ we show that the columns of $\bfPhi_{\smash{l}}$ form an ETF for $\bbF^b$.
To be precise, we show that the all off-diagonal entries of $\bfPhi_{\smash{l}}^*\bfPhi_{\smash{l}}^{}$ have modulus $1$,
while all diagonal entries have value $r$.
By~\eqref{equation.tensor definition of Steiner ETF},
$\bfPhi_{\smash{l}}$ can be regarded as a $b\times v$ array of submatrices of size $1\times(r+1)$,
namely
$\set{(\bfPhi_{\smash{l}})_{i,j}}_{i=1,}^{b}\,_{j=1}^{v}
=\set{\bff_{\smash{l}}^*\bfPi_{i,j}\bfG_1^*}_{i=1,}^{b}\,_{j=1}^{v}$.
Thus, the Gram matrix $\bfPhi_{\smash{l}}^*\bfPhi_{\smash{l}}^{}$ is a $v\times v$ array of submatrices of size $(r+1)\times(r+1)$.
Specifically, for any $j,j'=1,\dotsc,v$, the $(j,j')$th submatrix of $\bfPhi_{\smash{l}}^*\bfPhi_{\smash{l}}^{}$ is
\begin{equation}
\label{equation.Naimark complement of Steiner ETF 1}
(\bfPhi_{\smash{l}}^*\bfPhi_{\smash{l}}^{})_{j,j'}
=\sum_{i=1}^b(\bfPhi_{\smash{l}}^*)_{j,i}(\bfPhi_{\smash{l}}^{})_{i,j'}
=\sum_{i=1}^b(\bff_{\smash{l}}^*\bfPi_{i,j}\bfG_1^*)^*(\bff_{\smash{l}}^*\bfPi_{i,j'}\bfG_1^*).
\end{equation}
Here, for any $j=1,\dotsc,v$ and $s=1,\dotsc,r+1$, the $(1,s)$th entry of $\bff_{\smash{l}}^*\bfPi_{i,j'}\bfG_1^*$ is
\begin{equation}
\label{equation.Naimark complement of Steiner ETF 2}
(\bff_{\smash{l}}^*\bfPi_{i,j}\bfG_1^*)(1,s)
=\sum_{p=1}^{k}\sum_{q=1}^{r}\overline{\bff_{\smash{l}}(p)}\bfPi_{i,j}(p,q)\bfG_1^*(q,s)
=\sum_{p=1}^{k}\sum_{q=1}^{r}\bfPi_{i,j}(p,q)\overline{\bfF(p,l)}\overline{\bfG(s,q)}.
\end{equation}
Combining~\eqref{equation.Naimark complement of Steiner ETF 1}
and~\eqref{equation.Naimark complement of Steiner ETF 2} gives that for any $j,j'=1,\dotsc,v$ and $s,s'=1,\dotsc,r+1$,
the $(s,s')$th entry of the $(j,j')$th submatrix of $\bfPhi_{\smash{l}}^*\bfPhi_{\smash{l}}^{}$ is
\begin{align}
\nonumber
(\bfPhi_{\smash{l}}^*\bfPhi_{\smash{l}}^{})_{j,j'}(s,s')
&=\sum_{i=1}^b\overline{(\bff_{\smash{l}}^*\bfPi_{i,j}\bfG_1^*)(1,s)}(\bff_{\smash{l}}^*\bfPi_{i,j'}\bfG_1^*)(1,s')\\
\label{equation.Naimark complement of Steiner ETF 3}
&=\sum_{i=1}^b\sum_{p,p'=1}^{k}\sum_{q,q'=1}^{r}
\bfPi_{i,j}(p,q)\bfPi_{i,j'}(p',q')\bfF(p,l)\overline{\bfF(p',l)}\bfG(s,q)\overline{\bfG(s',q')}.
\end{align}
To proceed, recall that the definition of the BIBD permutation matrix
implies~\eqref{equation.definition of BIBD permutation matrix restated} where $\bfX$ is the $b\times v$ incidence matrix of the underlying $\BIBD(v,k,1,r,b)$.
In particular, if $j\neq j'$ then
\begin{equation*}
1
=(\bfX^\rmT\bfX)(j,j')
=\sum_{i=1}^{b}\bfX(i,j)\bfX(i,j')
=\sum_{i=1}^b\sum_{p,p'=1}^{k}\sum_{q,q'=1}^{r}\bfPi_{i,j}(p,q)\bfPi_{i,j'}(p',q').
\end{equation*}
Since $\bfPi$ is $\set{0,1}$-valued,
this means that when $j\neq j'$,
there exists exactly one choice of $(i,p,p',q,q')$ such that $\bfPi_{i,j}(p,q)\bfPi_{i,j'}(p',q')=1$.
(Here, $i$ corresponds to the unique block in the design that contains both the $j$th and $j'$th vertices, whereupon $(i,j)$ uniquely determines $(p,q)$ and $(i,j')$ uniquely determines $(p',q')$.)
As such, when $j\neq j'$, there is exactly one nonzero summand in~\eqref{equation.Naimark complement of Steiner ETF 3}, that is,
$(\bfPhi_{\smash{l}}^*\bfPhi_{\smash{l}}^{})_{j,j'}(s,s')
=\bfF(p,l)\overline{\bfF(p',l)}\bfG(s,q)\overline{\bfG(s',q')}$
for this unique choice of $(i,p,p',q,q')$.
In particular, if $j\neq j'$ then $\abs{(\bfPhi_{\smash{l}}^*\bfPhi_{\smash{l}}^{})_{j,j'}(s,s')}=1$ for all $s,s'$.

In the remaining case where $j= j'$, note that since each submatrix $\bfPhi_{i,j}$ contains at most one entry that has value $1$, $\bfPi_{i,j}(p,q)\bfPi_{i,j}(p',q')=1$ only when $p'=p$ and $q'=q$.
As such, when $j=j'$, \eqref{equation.Naimark complement of Steiner ETF 3} simplifies to
\begin{align*}
(\bfPhi_{\smash{l}}^*\bfPhi_{\smash{l}}^{})_{j,j}(s,s')
&=\sum_{i=1}^b\sum_{p=1}^{k}\sum_{q=1}^{r}
\bfPi_{i,j}(p,q)\bfF(p,l)\overline{\bfF(p,l)}\bfG(s,q)\overline{\bfG(s',q)}\\
&=\sum_{q=1}^{r}\bfG(s,q)\overline{\bfG(s',q)}\sum_{i=1}^b\sum_{p=1}^{k}\bfPi_{i,j}(p,q).
\end{align*}
Here, $\sum_{i=1}^b\sum_{p=1}^{k}\bfPi_{i,j}(p,q)$ is the sum of all of the entries in a column of the permutation matrix $\bfPi$, namely $1$.
When combined with the fact that $\bfG$ is a Hadamard matrix, this implies
\begin{equation*}
(\bfPhi_{\smash{l}}^*\bfPhi_{\smash{l}}^{})_{j,j}(s,s')
=\sum_{q=1}^{r}\bfG(s,q)\overline{\bfG(s',q)}
=(r+1)\bfI_{r+1}(s,s')-\bfG(s,r+1)\overline{\bfG(s',r+1)},
\end{equation*}
a quantity which equals $r$ when $s=s'$ and has modulus $1$ when $s\neq s'$.

For the final conclusions, note that if both $\bfF$ and $\bfG$ are (real) Hadamard matrices,
then the known half of the Hadamard conjecture implies that $k$ and $r+1$ are either $2$ or are divisible by $4$.
In particular, if $k\neq 2$, then $k$ and $r+1\geq k+1$ are divisible by $4$.
Since $b=\frac{vr}k$ is an integer and $r$ is odd, this implies that $4$ divides $v$,
contradicting the fact that $v=r(k-1)+1\equiv(-1)^2+1\equiv 2\bmod 4$.
Thus, if $\bfF$ and $\bfG$ are both Hadamard matrices, then $k=2$.
Conversely, if $k=2$ and there exists a Hadamard matrix of size $v$,
we note that there is a unique BIBD on $v$ vertices with this $k$ and with $\lambda=1$,
namely the $\BIBD(v,2,1,v-1,\frac12v(v-1))$ that consists of all $2$-element subsets of $[v]$.
Taking this BIBD, letting $\bfF$ be the canonical Hadamard matrix of size $2$,
and letting $\bfG$ be the given Hadamard matrix of size $v$, the resulting matrices $\bfPhi_1$ and $\widetilde{\bfPhi}_1$ are real flat matrices of size $\frac12v(v-1)\times v^2$ and $\frac12v(v+1)\times v^2$, respectively.
\end{proof}

For any $\BIBD(v,k,1,r,b)$ we note there always exists complex Hadamard matrices of size $k$ and $r+1$, such as DFTs.
In fact, if $k>2$ and there exists a Hadamard matrix of size $r+1$,
then taking $\bfG$ to be that matrix and letting $\bfF$ be a DFT (or any other complex Hadamard matrix whose first column is all ones),
the Steiner ETF $\bfPhi_1$ is real and its Naimark complement $\widetilde{\bfPhi}_1$ is complex.
Of course, any real ETF has a real Naimark complement.
In fact, if $\bfG$ is Hadamard and $\bfF$ is any real multiple of an orthogonal matrix whose first column is all ones,
then the first part of the proof of Theorem~\ref{theorem.Naimark complement of Steiner ETF} is still valid and~\eqref{equation.Naimark complement of Steiner ETF} is still a Naimark complement of $\bfPhi_1$.
However, in this case, $\bfPhi_l$ is not necessarily an ETF for $l=2,\dotsc,k$.
For example, the \textit{Fano plane} is a well-known $\BIBD(7,3,1,3,7)$,
and letting $\bfG$ be the canonical $4\times 4$ Hadamard matrix and letting either
\begin{equation*}
\bfF=\left[\begin{array}{rrr}
1&\sqrt{2}&0\smallskip\\
1&-\frac1{\sqrt{2}}&\frac{\sqrt3}{\sqrt2}\smallskip\\
1&-\frac1{\sqrt{2}}&-\frac{\sqrt3}{\sqrt2}
\end{array}\right]
\quad
\text{or}
\quad
\bfF=\left[\begin{array}{lll}
1&\,1&\,1\\
1&\omega&\omega^2\\
1&\omega^2&\omega
\end{array}\right],\ \omega=\exp(\tfrac{2\pi\rmi}3)
\end{equation*}
gives either a real or complex Naimark complement $\widetilde{\bfPhi}_1$ for the same $7\times 28$ real ETF $\bfPhi_1$.
In the second case however, $\bfPhi_2$ and $\bfPhi_3$ are themselves complex Steiner ETFs;
as we now discuss, this is important when trying to modify these ETFs in a way that makes them complex Hadamard.

\subsection{Naimark complements of Kirkman equiangular tight frames}

A BIBD is \textit{resolvable}---denoted an RBIBD---if its $b$ blocks can be arranged as $r$ collections of $\frac br=\frac vk$ blocks apiece---called \textit{parallel classes}---with each parallel class forming a partition of the vertex set.
For example, \eqref{equation.6x4 incidence matrix} is the incidence matrix of an $\RBIBD(4,2,1,3,6)$ since its six blocks can be arranged as three parallel classes---blocks one and two, blocks three and four, and blocks five and six---with each class giving a partition of the vertex set $\set{1,2,3,4}$.

In~\cite{JasperMF14}, it is shown that every Steiner ETF arising from a $\RBIBD(v,k,1,r,b)$ can be made flat by applying a scaled unitary operator to it.
Such ETFs are dubbed Kirkman ETFs, in honor of \textit{Kirkman's schoolgirl problem}, a foundational problem in combinatorial design regarding the existence of an $\RBIBD(15,3,1,7,35)$.
To be precise, let $\bfX$ be the incidence matrix of an $\RBIBD(v,k,1,r,b)$,
arranged without loss of generality as an $r\times 1$ array of submatrices of size $\frac vk\times v$, each corresponding to a parallel class.
For any $l=1,\dotsc,k$, and any possibly-complex Hadamard matrices $\bfF$ and $\bfG$ of size $k$ and $r+1$, respectively, let $\bfPhi_l$ be the corresponding Steiner ETF~\eqref{equation.tensor definition of Steiner ETF}.
Due to the structure of $\bfX$,
every column of $\bfPhi_l$ is a direct sum of $r$ vectors in $\bbF^{\frac vk}$, each having a single entry of modulus one with all other entries being zero.
Multiplying any such vector by a possibly-complex Hadamard matrix $\bfE$ of size $\frac vk$ produces a vector with all unimodular entries.
In particular, $(\bfI_r\otimes\bfE)\bfPhi_l$ is flat.
Moreover, since $(\bfI_r\otimes\bfE)$ is a scalar multiple of a unitary matrix,
the columns of $(\bfI_r\otimes\bfE)\bfPhi_l$ still form an ETF for $\bbF^b$.
When combined with Theorem~\ref{theorem.Naimark complement of Steiner ETF},
these ideas lead to the following result:

\begin{theorem}
\label{theorem.Naimark complement of Kirkman ETF}
Let $\bfPi$ be the permutation matrix of an $\RBIBD(v,k,1,r,b)$, and let $\bfE$, $\bfF$ and $\bfG$ be possibly-complex Hadamard matrices of size $\frac vk$, $k$ and $r+1$, respectively.
Write $\bfG=\left[\begin{array}{cc}\bfG_1&\bfg_2\end{array}\right]$ where $\bfG_1$ is $(r+1)\times r$.
For any $l=1,\dotsc,k$, let
\begin{equation*}
\bfPsi_{\smash{l}}
:=(\bfI_r\otimes\bfE)(\bfI_b^{}\otimes\bff_{\smash{l}}^*)\bfPi(\bfI_v\otimes\bfG_1^*),
\end{equation*}
where $\bff_{\smash{l}}$ is the $l$th column of $\bfF$.
Then, the $v(r+1)$ columns of each $\bfPsi_{\smash{l}}$ form a flat ETF for $\bbF^b$.
Moreover, the row spaces of the matrices $\set{\bfPsi_{\smash{l}}}_{l=1}^{k}$ are mutually orthogonal and the matrix
\begin{equation*}
\widetilde{\bfPsi}_1=\left[\begin{array}{c}
\bfPsi_2\\
\vdots\\
\bfPsi_k\\
(\bfE\otimes\bfF)(\bfI_v\otimes\bfg_2^*)
\end{array}\right]
\end{equation*}
is a flat Naimark complement of $\bfPsi_1$,
meaning $\bfPsi_1$ is a possibly-complex Hadamard ETF.

In particular, if there exists a Hadamard matrix of size $u$ then there exists a Hadamard ETF with parameters $(d,n)=(u(2u-1),4u^2)$.
\end{theorem}

\begin{proof}
As noted above,
for each $l=1,\dotsc,j$,
the columns of $\bfPsi_l=(\bfI_r\otimes\bfE)\bfPhi_l$ form a flat ETF for $\bbF^b$.
Moreover, for each $l=1,\dotsc,j$, the row space of $\bfPsi_l$ equals that of $\bfPhi_l$.
By Theorem~\ref{theorem.Naimark complement of Steiner ETF}, these row spaces are mutually orthogonal.
They are also orthogonal to the row space of $\sqrt{k}(\bfI_v\otimes\bfg_2^*)$,
which equals that of $(\bfE\otimes\bfF)(\bfI_v\otimes\bfg_2^*)$.
Here, $\bfE\otimes\bfF$ is a possibly-complex Hadamard matrix of size $v$,
implying that $(\bfE\otimes\bfF)(\bfI_v\otimes\bfg_2^*)$ has unimodular entries and orthogonal rows.
Altogether, these facts imply that $\widetilde{\bfPsi}_1$ is indeed a flat Naimark complement for $\bfPsi_1$.
By Proposition~\ref{proposition.Hadamard ETF characterization}, this means $\bfPhi_1$ is a possibly-complex Hadamard ETF.

For the final conclusion, assume there exists a Hadamard matrix $\bfE$ of size $u$,
and let $v=2u$.
Since $v$ is even, the unique $\BIBD(v,2,1,v-1,\frac12v(v-1))$ is resolvable using the well-known \textit{round-robin tournament schedule}.
As such, letting $\bfF$ be the canonical Hadamard matrix of size $2$,
and noting that $\bfG=\bfE\otimes\bfF$ is a Hadamard matrix of size $2u=v=r+1$,
the columns of $\bfPsi_1$ and $\widetilde{\bfPhi}_1$ form Naimark complementary real flat ETFs with parameters $(d,n)=(\frac12v(v-1),v^2)=(u(2u-1),4u^2)$ and $(n-d,n)=(\frac12v(v+1),v^2)=(u(2u+1),4u^2)$, respectively.
Thus, both ETFs are Hadamard.
\end{proof}

Generally speaking, the significance of Theorem~\ref{theorem.Naimark complement of Kirkman ETF} is that it shows that the fundamental idea of harmonic ETFs---that ETFs can be constructed by carefully extracting rows from a possibly-complex Hadamard matrix---can be successful even when the matrix is not the character table of an abelian group.
For example, taking $u=12$ gives a Hadamard matrix of size $576$ from which $276$ special rows can be extracted to form a real ETF; in contrast, the character table of any abelian group of order $576=2^6 3^2$ contains cube roots of unity.
This realization hopefully better informs searches for new ETFs in the future.

That said, the most immediate contribution of Theorem~\ref{theorem.Naimark complement of Kirkman ETF} is that it gives a new infinite family of real flat ETFs.
Indeed, prior to this result, the only known real flat ETFs with parameters of the form $(u(2u+1),4u^2)$ had $u=2^e$ for some $e\geq0$;
such ETFs arise, for example, from the complements of McFarland difference sets in \smash{$\bbZ_2^{2(e+1)}$}~\cite{DingF07}.
We now know they exist whenever there exists a Hadamard matrix of size $u$.
An infinite number of these ETFs are new: for example,
there are an infinite number of prime powers $q$ that are congruent to $1$ modulo $4$, and for each of these, Paley's method yields a Hadamard matrix of size $u=2(q+1)$, which is not a power of $2$.

\section{Real flat equiangular tight frames and quasi-symmetric designs}

In this section, we combine ideas from~\cite{McGuire97} and~\cite{JasperMF14} to give a new characterization of real flat ETFs in terms of particular types of BIBDs known as \textit{quasi-symmetric designs} (QSDs).
As we shall see, this characterization further generalizes to a particular class of real ETFs that are not necessarily flat, namely those that arise from the \textit{block graphs} of QSDs~\cite{GoethalsS70} via a well-known equivalence between real ETFs and certain \textit{strongly regular graphs} (SRGs)~\cite{StrohmerH03,HolmesP04,Waldron09}.

To be precise, a $\BIBD(v,r,\lambda,r,b)$ with $b>v$ is a QSD with \textit{block intersection numbers} $y>x\geq 0$ if the cardinality of the intersection of any two distinct blocks is either $x$ or $y$.
That is, a BIBD with $b>v$ is a $\QSD(v,k,\lambda,r,b,x,y)$ if its $b\times v$ incidence matrix $\bfX$ satisfies
\begin{equation}
\label{equation.adjacency matrix of block graph}
\bfX\bfX^\rmT
=(k-x)\bfI+(y-x)\bfA+x\bfJ
\end{equation}
for some $\set{0,1}$-valued matrix $\bfA$.
Here, $\bfA$ is the adjacency matrix of a graph known as the \textit{block graph} of the QSD.
One simple example of a QSD is a BIBD with $b>v$ and $\lambda=1$:
in such a design, any two distinct blocks have at most one vertex in common,
and so it is a QSD with $x=0$ and $y=1$.
Another common way to construct a QSD is to take the \textit{complementary design} of another QSD,
namely to consider the incidence matrix $\bfJ-\bfX$ instead of $\bfX$.

In our arguments below, we use the well-known fact that any QSD realizes both of its intersection numbers $\set{x,y}$,
namely that its corresponding block graph is neither complete nor empty.
Indeed, if we instead have a BIBD in which any two distinct blocks intersect in exactly $x$ vertices where $x<k$,
then $\bfX\bfX^\rmT=(k-x)\bfI+x\bfJ$ is positive-definite,
contradicting the underlying assumption that $b>v$.
(BIBDs with $b=v$ are instead called \textit{symmetric designs}.)
We shall also make use of the fact that the parameters of a QSD are dependent~\cite{Shrikhande07}, satisfying
\begin{equation}
\label{equation.QSD parameter relationship}
k(r-1)(x+y-1)-xy(b-1)=k(k-1)(\lambda-1).
\end{equation}

It is known that the block graph of any QSD is an SRG~\cite{GoethalsS70,Shrikhande07}.
To elaborate, a graph on $b$ vertices is an $\SRG(b,a,c,\mu)$ if any vertex has $a$ neighbors,
any two adjacent vertices have $c$ neighbors in common, and any two nonadjacent vertices have $\mu$ neighbors in common,
namely when its adjacency matrix $\bfA$ satisfies $\bfA\bfone=a\bfone$ and $\bfA^2=(c-\mu)\bfA+(a-\mu)\bfI+\mu\bfJ$.
Using~\eqref{equation.BIBD properties},
it is straightforward to show that the matrix $\bfA$ defined by~\eqref{equation.adjacency matrix of block graph} satisfies such a relationship where
\begin{equation}
\label{equation.QSD SRG parameters}
a=\tfrac{k(r-1)-x(b-1)}{y-x},
\quad
\theta_1=\tfrac{(r-\lambda)-(k-x)}{y-x},
\quad
\theta_2=-\tfrac{k-x}{y-x},
\quad
c=a+\theta_1+\theta_2+\theta_1\theta_2,
\quad
\mu=a+\theta_1\theta_2.
\end{equation}
As with any SRG, these parameters are dependent~\cite{Brouwer07}, satisfying
\begin{equation}
\label{equation.SRG parameter relationship}
a(a-c-1)=\mu(b-a-1).
\end{equation}

It is also known that real ETFs are equivalent to a certain class of SRGs~\cite{StrohmerH03,HolmesP04,Waldron09}.
In particular, as detailed in~\cite{FickusJMP17},
real ETFs correspond to SRGs whose parameters $(b,a,c,\mu)$ satisfy $a=2\mu$;
in this case, the corresponding real ETF has parameters
\begin{equation}
\label{equation.ETF from SRG parameters}
n=b+1,
\quad
d=\tfrac{b+1}{2}\Bigbracket{1+\tfrac{b-2a-1}{\sqrt{(b-2a-1)^2+4b}}}.
\end{equation}

Putting all of these facts together,
it is natural to ask which class of QSDs leads to SRGs which equate to real ETFs.
By searching tables of known SRGs~\cite{Brouwer07,Brouwer17},
we find some instances where this occurs.
For example, there exists a $\BIBD(15,3,1,7,35)$;
having $\lambda=1$, this design is also a QSD with $(x,y)=(0,1)$.
By~\eqref{equation.QSD SRG parameters},
the corresponding block graph is an $\SRG(35,18,9,9)$.
Since the parameters of this SRG satisfy $a=18=2(9)=2\mu$,
it indeed corresponds to a real ETF;
by~\eqref{equation.ETF from SRG parameters}, this ETF has parameters $(d,n)=(15,36)$.
Note here that $d=15=v$.
This is not a coincidence: using~\eqref{equation.BIBD parameters relationship}, \eqref{equation.QSD parameter relationship}, \eqref{equation.QSD SRG parameters}, \eqref{equation.SRG parameter relationship} and~\eqref{equation.ETF from SRG parameters},
the interested reader can show that $d=v$ for any real ETF that arises from an $\SRG(b,a,c,\mu)$ with $a=2\mu$ that itself arises from a $\QSD(v,k,\lambda,r,b,x,y)$.
This inspires us to seek a relationship between the $b\times v$ incidence matrix $\bfX$ of such a QSD, and the $v\times(b+1)$ synthesis operator of the ETF synthesis operator it generates.
In particular, in Theorem~\ref{theorem.QSD} we characterize when there exists scalars $\delta$ and $\varepsilon$ such that
$\bfPhi=\left[\begin{array}{cc}\bfone&\delta\bfJ+\varepsilon\bfX^\rmT\end{array}\right]$ is the synthesis operator of an ETF.

Combining results from the existing literature reveals a second connection between QSDs and ETFs.
In particular, \cite{McGuire97} establishes an equivalence between a certain class of QSDs and  self-complementary binary codes that achieve equality in the Grey-Rankin bound.
In~\cite{JasperMF14}, it was shown that these same codes are equivalent to real flat ETFs.
Together, these two results imply an equivalence between real flat ETFs and a particular class of QSDs.
As we shall see, it turns out that this equivalence is a special case of that described in the previous paragraph, namely when the scalars $\delta$ and $\varepsilon$ can be chosen to be $1$ and $-2$, respectively.
The analysis here is delicate:
a $\QSD(15,3,1,7,35,0,1)$ exists and yields a real ETF with $(d,n)=(15,36)$.
However, as we shall see, such an ETF cannot be flat.
Meanwhile, a $\QSD(6,2,1,5,15,0,1)$ exists and yields a real flat ETF with $(d,n)=(6,16)$.

\begin{theorem}
\label{theorem.QSD}
For any $\delta,\varepsilon\in\bbR$ and any $\set{0,1}$-valued $b\times v$ matrix $\bfX$ with $b>v>1$, let
\begin{equation*}
\bfPhi=\left[\begin{array}{cc}\bfone&\delta\bfJ+\varepsilon\bfX^\rmT\end{array}\right].
\end{equation*}
There exists a choice of $\delta,\varepsilon\in\bbR$ such that the columns $\set{\bfphi_j}_{j=1}^{b+1}$ of $\bfPhi$ form an ETF for $\bbR^v$ with $\ip{\bfphi_1}{\bfphi_j}>0$ for all $j$ if and only if $\bfX$ is the incidence matrix of a $\QSD(v,k,\lambda,r,b,x,y)$ with $0<k<v$ whose parameters satisfy the following relationships:
\begin{equation}
\label{equation.QSD parameters}
w=\bigbracket{\tfrac{v(b+1-v)}{b}}^{\frac12},
\quad
r=\tfrac{bk}{v},
\quad
\lambda=\tfrac{r(k-1)}{v-1},
\quad
x=k-\tfrac{(v+w)(r-\lambda)}{b+1},
\quad
y=k-\tfrac{(v-w)(r-\lambda)}{b+1}.
\end{equation}
Specifically, there are two choices for $(\delta,\varepsilon)$ here:
\begin{equation}
\label{equation.QSD alpha beta}
\delta
=\tfrac1{v}[w\pm k(\tfrac{b+1}{r-\lambda})^{\frac12}],
\qquad
\varepsilon
=\tfrac1k(w-\delta v)
=\mp(\tfrac{b+1}{r-\lambda})^{\frac12}.
\end{equation}

In particular, if $n-1>d>1$ and $\bfPhi$ is any $\set{\pm1}$-valued $d\times n$ matrix whose rows and columns have been signed to assume without loss of generality that $\bfPhi=\left[\begin{array}{cc}\bfone&\bfJ-2\bfX^\rmT\end{array}\right]$ where $\bfX$ is $\set{0,1}$-valued and $\ip{\bfphi_1}{\bfphi_j}\geq 0$ for all $j$,
then the columns of $\bfPhi$ form an ETF for $\bbR^d$ if and only if $\bfX$ is the incidence matrix of a $\QSD(v,k,\lambda,r,b,x,y)$ with $v=d$, $b=n-1$, and $k=\tfrac{v-w}{2}$ where $w$, $r$ and $\lambda$ are given by~\eqref{equation.QSD parameters} and $x=\tfrac{v-3w}4$, $y=\tfrac{v-w}4$.
In this case, $d$ and $w$ are necessarily even, and $n$ is necessarily divisible by $4$.
\end{theorem}

\begin{proof}
To simplify notation, let $\bfZ=\bfX^\rmT$.
For any $\delta,\varepsilon\in\bbR$,
the fact that $\norm{\bfphi_1}^2=\norm{\bfone}^2=v$ implies that $\set{\bfphi_j}_{j=1}^{b+1}$ is an ETF for $\bbR^v$ if and only if
\begin{equation}
\label{equation.proof of QSD 1}
\abs{\ip{\bfphi_j}{\bfphi_{j'}}}
=\left\{\begin{array}{cl}v,&j=j',\\w,&j\neq j',\end{array}\right.
\qquad
\forall j,j'=1,\dotsc,b+1,
\end{equation}
where $w$, as defined in~\eqref{equation.QSD parameters}, is obtained by scaling the Welch bound for $n=b+1$ and $d=v$ by a factor of $\norm{\bfphi_1}^2=v$.
Letting $\set{\bfz_j}_{j=1}^{b}$ denote the columns of $\bfZ$,
we have $\bfphi_{j+1}=\delta\bfone+\varepsilon\bfz_j$ for all $j=1,\dotsc,b$.
Having the additional property that $\ip{\bfphi_1}{\bfphi_j}>0$ for all $j$ equates to having
\begin{equation*}
w
=\ip{\bfphi_1}{\bfphi_{j+1}}
=\ip{\bfone}{\delta\bfone+\varepsilon\bfz_j}
=\delta v+\varepsilon\ip{\bfone}{\bfz_j},
\qquad
\forall j=1,\dotsc,b.
\end{equation*}
If $\varepsilon=0$, the columns of $\bfPhi$ are collinear,
meaning they are not a tight frame for $\bbR^v$ since $v>1$.
When $\varepsilon\neq0$, the above equation gives that \smash{$\ip{\bfone}{\bfz_j}=\frac1{\varepsilon}(w-\delta v)$} for all $j=1,\dotsc,b$, meaning in particular that each column of $\bfZ$ contains exactly $k=\frac1{\varepsilon}(w-\delta v)$ ones.
Here, $k$ is an integer satisfying $0<k<v$ since having either $k=0$ or $k=v$ again implies that the columns of $\bfPhi$ are collinear.
Moreover, in this case, the vectors $\set{\bfphi_j}_{j=1}^{b+1}=\set{\bfone}\cup\set{\delta\bfone+\varepsilon\bfz_j}_{j=1}^{b}$
satisfy~\eqref{equation.proof of QSD 1} if and only if for every $j,j'=1,\dotsc,b$, $j\neq j'$ we have
\begin{align}
\label{equation.proof of QSD 2}
v&=\norm{\delta\bfone+\varepsilon\bfz_j}^2
=\delta^2\norm{\bfone}^2+2\delta\varepsilon\ip{\bfone}{\bfz_j}+\varepsilon^2\norm{\bfz_j}^2
=\delta^2v+(2\delta\varepsilon+\varepsilon^2)k,\\
\label{equation.proof of QSD 3}
w
&=\abs{\ip{\delta\bfone+\varepsilon\bfz_j}{\delta\bfone+\varepsilon\bfz_{j'}}}
=\abs{\delta^2 v+2\delta\varepsilon k+\varepsilon^2\ip{\bfz_j}{\bfz_{j'}}}.
\end{align}

To summarize, there exists $\delta,\varepsilon\in\bbR$ such that
$\set{\bfphi_j}_{j=1}^{b+1}=\set{\bfone}\cup\set{\delta\bfone+\varepsilon\bfz_j}_{j=1}^{b}$
is an ETF for $\bbR^v$ with $\ip{\bfphi_1}{\bfphi_j}>0$ for all $j$ if and only if there exists
$\delta\in\bbR$ and an integer $k$, $0<k<v$ such that letting $\varepsilon=\tfrac1k(w-\delta v)$ we have that $\delta$ and $\varepsilon$ satisfy~\eqref{equation.proof of QSD 2},
that $\set{\bfz_j}_{j=1}^{b}$ satisfies $\ip{\bfone}{\bfz_j}=k$ for all $j$,
and that $\set{\bfz_j}_{j=1}^{b}$ satisfies~\eqref{equation.proof of QSD 3} for all $j\neq j'$.
Continuing, note that since $\varepsilon=\frac{w-\delta v}{k}$, \eqref{equation.proof of QSD 2} becomes
\begin{equation*}
v
=\delta^2v+(2\delta\varepsilon+\varepsilon^2)k
=\delta^2v+2\delta(w-\delta v)+\tfrac{1}{k}(w-\delta v)^2,
\end{equation*}
which is equivalent to having
\begin{equation}
\label{equation.proof of QSD 4}
0
=\tfrac12 v\delta^2-w\delta-\tfrac{vk-w^2}{2(v-k)}.
\end{equation}
To find the roots of this equation, it helps to define $r$ and $\lambda$ as in~\eqref{equation.QSD parameters};
since $0<k<v$, these definitions imply $0\leq\lambda<r<b$ with
\begin{equation*}
0
<r-\lambda
=(1-\tfrac{k-1}{v-1})r
=\tfrac{v-k}{v-1}r
=\tfrac{bk(v-k)}{v(v-1)}.
\end{equation*}
This fact and the definition of $w$ imply that the discriminant of \eqref{equation.proof of QSD 4} is the positive quantity
\begin{equation}
\label{equation.proof of QSD 5}
w^2
+v\tfrac{vk-w^2}{v-k}
=\tfrac{k(v^2-w^2)}{v-k}
=\tfrac{k}{v-k}\bigbracket{v^2-\tfrac{v}{b}(b+1-v)}
=\tfrac{kv(v-1)(b+1)}{b(v-k)}
=k^2\tfrac{b+1}{r-\lambda},
\end{equation}
meaning~\eqref{equation.proof of QSD 4} has two real roots, each of which leads to a corresponding value of $\varepsilon$,
namely those paired values of $\delta$ and $\varepsilon$ given in~\eqref{equation.QSD alpha beta}.

Having characterized when~\eqref{equation.proof of QSD 2} is satisfied, we turn to~\eqref{equation.proof of QSD 3}:
in light of~\eqref{equation.QSD alpha beta} and~\eqref{equation.proof of QSD 4},
\begin{align}
\nonumber
\delta^2 v+2\delta\varepsilon k+\varepsilon^2\ip{\bfz_j}{\bfz_{j'}}
&=\delta^2 v+2\delta(w-\delta v)+\tfrac{b+1}{r-\lambda}\ip{\bfz_j}{\bfz_{j'}}\\
\nonumber
&=-2(\tfrac12 v\delta^2-w\delta)+\tfrac{b+1}{r-\lambda}\ip{\bfz_j}{\bfz_{j'}}\\
\label{equation.proof of QSD 6}
&=-\tfrac{vk-w^2}{v-k}+\tfrac{b+1}{r-\lambda}\ip{\bfz_j}{\bfz_{j'}}.
\end{align}
Repurposing~\eqref{equation.proof of QSD 5} further gives
\begin{equation*}
\tfrac{vk-w^2}{v-k}
=\tfrac{v[v-(v-k)]-w^2}{v-k}
=\tfrac{v^2-w^2}{v-k}-v
=k\tfrac{b+1}{r-\lambda}-v.
\end{equation*}
Substituting this into \eqref{equation.proof of QSD 6},
we see that \eqref{equation.proof of QSD 3} is satisfied when
\begin{equation*}
v-k\tfrac{b+1}{r-\lambda}+\tfrac{b+1}{r-\lambda}\ip{\bfz_j}{\bfz_{j'}}
=\delta^2 v+2\delta\varepsilon k+\varepsilon^2\ip{\bfz_j}{\bfz_{j'}}
\in\set{-w,w},
\end{equation*}
namely when $\ip{\bfz_j}{\bfz_{j'}}\in\set{x,y}$ where $x$ and $y$ are defined in \eqref{equation.QSD alpha beta}.

To summarize,
there exists $\delta,\varepsilon\in\bbR$ such that
$\set{\bfphi_j}_{j=1}^{b+1}=\set{\bfone}\cup\set{\delta\bfone+\varepsilon\bfz_j}_{j=1}^{b}$ is an ETF for $\bbR^v$ with $\ip{\bfphi_1}{\bfphi_j}>0$ for all $j$ if and only if there exists an integer $k$, $0<k<v$ such that $\ip{\bfone}{\bfz_j}=k$ for all $j$
and such that $\ip{\bfz_j}{\bfz_{j'}}\in\set{x,y}$ for all $j\neq j'$,
under the definitions given in~\eqref{equation.QSD alpha beta}.
In this case, there are two choices for $(\delta,\varepsilon)$, namely the values given in \eqref{equation.QSD alpha beta}.
This characterization is complete.
However, it differs from the characterization given in the statement of the result since we have not yet used the fact that any vectors that meet the Welch bound are necessarily tight,
a fact which pertains to the rows of $\bfZ$.
In particular, under the above hypotheses, we have $\bfPhi\bfPhi^*=\alpha\bfI$ where $\alpha v=\Tr(\alpha\bfI)=\Tr(\bfPhi\bfPhi^*)=\Tr(\bfPhi^*\bfPhi)=\sum_{j=1}^{b+1}\norm{\bfphi_j}^2=(b+1)v$, implying $\alpha=b+1$.
As such,
\begin{equation*}
(b+1)\bfone
=\bfPhi\bfPhi^*\bfone
=\sum_{j=1}^{b+1}\ip{\bfphi_j}{\bfone}\bfphi_j
=v\bfone+w\sum_{j=1}^{b}(\delta\bfone+\varepsilon\bfz_j)
=(v+bw\delta)\bfone+w\varepsilon\sum_{j=1}^{b}\bfz_j.
\end{equation*}
This implies that each row of $\bfZ$ sums to \smash{$r':=\tfrac1{w\varepsilon}(b+1-v-bw\delta)$};
since $\bfZ$ is a $\set{0,1}$-valued $v\times b$ matrix whose columns sum to $k$,
we have $vr'=bk=vr$ and so $r'=r$.
This in turn implies that for any $i,i'=1,\dotsc,v$, $i\neq i'$,
\begin{equation*}
0
=(\bfPhi\bfPhi^*)(i,i')
=1+\sum_{j=1}^{b}[\delta+\varepsilon\bfZ(i,j)][\delta+\varepsilon\bfZ(i',j)]
=1+\delta^2 b+2\delta\varepsilon r+\varepsilon^2(\bfZ\bfZ^\rmT)(i,i'),
\end{equation*}
meaning that the dot product of any two distinct rows of $\bfZ$ is $\lambda'=-\tfrac1{\varepsilon^2}(1+\delta^2 b+2\delta\varepsilon r)$.
As such, $\bfZ$ is the $v\times b$ incidence matrix of a $\BIBD(v,k,\lambda')$ implying $\lambda'(v-1)=k(r-1)=\lambda(v-1)$ and so $\lambda'=\lambda$.
(The interested reader can also use~\eqref{equation.QSD alpha beta} to directly show that $r'=r$ and $\lambda'=\lambda$.)

As such, if there exists $\delta,\varepsilon\in\bbR$ such that
\smash{$\set{\bfphi_j}_{j=1}^{b+1}=\set{\bfone}\cup\set{\delta\bfone+\varepsilon\bfz_j}_{j=1}^{b}$} is an ETF for $\bbR^v$ with $\ip{\bfphi_1}{\bfphi_j}>0$,
then \smash{$\bfX=\bfZ^\rmT$} is the incidence matrix of a $\QSD(v,k,\lambda,r,b,x,y)$ with $0<k<v$.
Conversely, if \smash{$\bfX=\bfZ^\rmT$} is the incidence matrix of such a design,
then $\ip{\bfone}{\bfz_j}=k$ for all $j$ and $\ip{\bfz_j}{\bfz_{j'}}\in\set{x,y}$ for all $j\neq j'$, implying that $\set{\bfphi_j}_{j=1}^{b+1}=\set{\bfone}\cup\set{\delta\bfone+\varepsilon\bfz_j}_{j=1}^{b}$ is such an ETF provided $\delta$ and $\varepsilon$ are chosen according to~\eqref{equation.QSD alpha beta}.

For the next set of conclusions, let $\bfPhi$ be any $\set{\pm1}$-valued $n\times d$ matrix where $n-1>d>1$.
To determine when the columns $\set{\bfphi_j}_{j=1}^{n}$ of $\bfPhi$ form an ETF for $\bbR^d$,
note that by signing the rows of $\bfPhi$ we can assume without loss of generality that $\bfphi_1=\bfone$.
Moreover, by signing the columns of $\bfPhi$ we can further assume without loss of generality that $\ip{\bfphi_1}{\bfphi_j}\geq0$ for all $j$.
Any such matrix is of the form
$\bfPhi=\left[\begin{array}{cc}\bfone&\bfJ-2\bfZ\end{array}\right]$
where $\bfZ$ is a $\set{0,1}$-valued $v\times b$ matrix where $v=d$, $b=n-1$.
By what we have already seen,
$\set{\bfphi_j}_{j=1}^{n}=\set{\bfone}\cup\set{\bfone-2\bfz_j}_{j=1}^{b}$
is an ETF for $\bbR^d$ if and only if there exists an integer $k$, $0<k<v$ such that $\bfZ^\rmT$ is the incidence matrix of a QSD whose remaining parameters are given by~\eqref{equation.QSD parameters},
provided~\eqref{equation.QSD alpha beta} allows $(\delta,\varepsilon)=(1,-2)$, that is, if and only if
\begin{equation*}
1
=\delta
=\tfrac1{v}[w+k(\tfrac{b+1}{r-\lambda})^{\frac12}],
\qquad
-2
=\varepsilon
=-(\tfrac{b+1}{r-\lambda})^{\frac12}.
\end{equation*}
This occurs precisely when $\tfrac{b+1}{r-\lambda}=4$ and $v=w+2k$.
These two conditions are redundant:
if $v=w+2k$ then again repurposing~\eqref{equation.proof of QSD 5} gives
\begin{equation*}
\tfrac{b+1}{r-\lambda}
=\tfrac1{k^2}\tfrac{k(v^2-w^2)}{v-k}
=\tfrac{v^2-w^2}{k(v-k)}
=\tfrac{(w+2k)^2-w^2}{k(w+k)}
=4.
\end{equation*}
As such, the columns of $\bfPhi=\left[\begin{array}{cc}\bfone&\bfJ-2\bfZ\end{array}\right]$ form an ETF for $\bbR^d$ if and only if letting $k=\tfrac{v-w}{2}$ we have that $\bfZ^\rmT$ is the incidence matrix of a QSD whose remaining parameters are given by~\eqref{equation.QSD parameters}.
Here, since $k=\tfrac{v-w}{2}$ and $\tfrac{b+1}{r-\lambda}=4$, our expressions for $x$ and $y$ simplify to
\begin{equation*}
x=k-\tfrac{(v+w)(r-\lambda)}{b+1}=\tfrac{v-w}{2}-\tfrac{v+w}4=\tfrac{v-3w}4,
\qquad
y=k-\tfrac{(v-w)(r-\lambda)}{b+1}=\tfrac{v-w}{2}-\tfrac{v-w}4=\tfrac{v-w}4.
\end{equation*}
In this case, we have $n=b+1=4(r-\lambda)$ is divisible by $4$.
Moreover, since $x$ and $y$ are both integers,
then so are $y-x=\frac{w}{2}$ and $y+x=\frac{v}2-w$,
implying $w$ and $v$ are even integers.
\end{proof}

We note that Theorem~\ref{theorem.QSD} specifically excludes ETFs in which $n=d+1$, that is, regular simplices.
We did this because regular simplices are already well understood,
and because the corresponding incidence matrices are square.
In particular, as discussed in Section~2,
the existence of a flat regular simplex for $\bbR^d$ is equivalent to that of a Hadamard matrix of size $d+1$.
In this case, the arguments of the proof of Theorem~\ref{theorem.QSD} do not produce a QSD,
but rather a symmetric design with $k=\frac{v-1}2$ and $x=\lambda=\frac{v-3}4$.
For example, to put the real flat tetrahedron of~\eqref{equation.3x4 real flat regular simplex}
in the form where Theorem~\ref{theorem.QSD} applies, we negate all but its first column,
giving the following real flat ETF and corresponding symmetric design:
\begin{equation*}
\bfPhi
=\left[\begin{array}{rrrr}
 1& 1&-1& 1\\
 1&-1& 1& 1\\
 1& 1& 1&-1\\
 \end{array}\right],
\quad
\bfX^\rmT
=\left[\begin{array}{rrrr}
 0& 1& 0\\
 1& 0& 0\\
 0& 0& 1\\
 \end{array}\right].
\end{equation*}
In this situation, no two distinct blocks intersect in $y=\frac{v-1}4$ points,
giving no guarantee that this is an integer and by extension no guarantee that $v$ is even.
In fact, since in this case there necessarily exists a Hadamard matrix of size $v+1$, we have that $v$ is necessarily odd.
This explains how Theorem~\ref{theorem.QSD} is consistent with the dichotomy of the statement of Theorem~A of~\cite{McGuire97}:
for any real flat ETF for $\bbR^d$,
we either have $d$ is odd---meaning the ETF is a regular simplex arising from a Hadamard matrix of size $d+1$---or $d$ is even, meaning the ETF equates to a QSD whose parameters are given by Theorem~\ref{theorem.QSD}.
This in particular means that not all of the ETFs produced by Theorem~\ref{theorem.QSD} are flat:
applying it to a $\QSD(15,3,1,7,35,0,1)$ gives a real ETF with $(d,n)=(15,36)$, but since $15$ is odd, this ETF is not flat.

We also mention that if a real ETF arises from the block graph of a QSD,
then Theorem~\ref{theorem.QSD} can be applied to that QSD to produce an explicit ETF with those parameters.
To be precise, for any QSD whose block graph's SRG parameters satisfy $a=2\mu$,
one can use~\eqref{equation.BIBD parameters relationship}, \eqref{equation.QSD parameter relationship}, \eqref{equation.QSD SRG parameters} and \eqref{equation.SRG parameter relationship} to show that $y+x=2[k-\tfrac{v(r-\lambda)}{b+1}]$ and $(y-x)^2=\bigbracket{\tfrac{2w(r-\lambda)}{b+1}}^2$.
(These calculations are nontrivial, and we performed them with the aid of a computer algebra system.)
This implies $y$ and $x$ are of the form given in~\eqref{equation.QSD parameters},
meaning Theorem~\ref{theorem.QSD} can be applied to that QSD to produce an ETF with parameters $(d,n)=(v,b+1)$.
As noted above, these parameters match those of the ETF that arises from the QSD's strongly regular block graph via~\eqref{equation.ETF from SRG parameters}.
In particular, one can show that a QSD yields a real ETF via its block graph if and only if its parameters satisfy~\eqref{equation.QSD parameters}.

\subsection{Quasi-symmetric designs from Kirkman ETFs and their complements.}

We now combine Theorem~\ref{theorem.QSD} with previously-known results to produce other new results.

\begin{corollary}
\label{corollary.existence of real flat}
Let $u$ be an integer with $u\geq 2$.
There exists a real flat ETF with parameters $(d,n)=(u(2u-1),4u^2)$ if and only if there exists a QSD with parameters
\begin{equation}
\label{equation.special QSD parameters -}
(v,k,\lambda,r,b,x,y)
=(2u^2-u,\ u^2-u,\ u^2-u-1,\ 2u^2-u-1,\ 4u^2-1,\ \tfrac{u(u-2)}2,\ \tfrac{u(u-1)}2).
\end{equation}
Such ETFs exist whenever there exists a Hadamard matrix of size $u$,
or alternatively, whenever there exists a Hadamard matrix of size $2u$ and $u-2$ mutually orthogonal Latin squares (MOLS) of size $2u$, such as when $u=6$.

Similarly, there exists a real flat ETF with parameters $(d,n)=(u(2u+1),4u^2)$ if and only if there exists a QSD with parameters
\begin{equation}
\label{equation.special QSD parameters +}
(v,k,\lambda,r,b,x,y)
=(2u^2+u,\ u^2,\ u^2-u,\ 2u^2-u,\ 4u^2-1,\ \tfrac{u(u-1)}2,\ \tfrac{u^2}2).
\end{equation}
Such ETFs exist whenever there exists a Hadamard matrix of size $u$,
or alternatively, whenever there exists a Hadamard matrix of size $2u$ and $u-1$ MOLS of size $2u$, such as when $u=6$.

In either case, $u$ is necessarily even.
Also, when there exists a Hadamard matrix of size $u$, these two ETFs can be chosen to be Naimark complements.
\end{corollary}

\begin{proof}
If there exists a real flat ETF with parameters $(d,n)=(u(2u-1),4u^2)$,
then applying Theorem~\ref{theorem.QSD} to it produces a QSD with parameters~\eqref{equation.special QSD parameters -} and $w=u$.
Conversely, applying Theorem~\ref{theorem.QSD} to any QSD with parameters~\eqref{equation.special QSD parameters -} produces a real flat ETF with parameters $(d,n)=(u(2u-1),4u^2)$.
Similarly, there exists a real flat ETF with parameters $(u(2u+1),4u^2)$ if and only if there exists a QSD with parameters~\eqref{equation.special QSD parameters +};
here we again have $w=u$.
In either case, note Theorem~\ref{theorem.QSD} requires $w=u$ to be even.

If there exists a Hadamard matrix of size $u$,
\cite{Bracken06} gives a QSD with parameters~\eqref{equation.special QSD parameters -}.
For any such $u$,~\cite{JasperMF14} gives an independent method for constructing a real flat Kirkman ETF with parameters $(u(2u-1),4u^2)$.
The real flat Naimark complement of it constructed in Theorem~\ref{theorem.Naimark complement of Kirkman ETF} has parameters $(u(2u+1),4u^2)$.
As noted above, that ETF implies the existence of a QSD with parameters~\eqref{equation.special QSD parameters +}.
If we instead have a Hadamard matrix of size $2u$,
then Theorems~1 and~2 of~\cite{BrackenMW06} give QSDs with parameters~\eqref{equation.special QSD parameters -} and~\eqref{equation.special QSD parameters +} provided we also have $u-2$ or $u-1$ MOLS of size $2u$, respectively.
These conditions hold, for example, when $u=6$.
\end{proof}

As seen from this proof, the true novelty of Corollary~\ref{corollary.existence of real flat} is the existence of QSDs with parameters~\eqref{equation.special QSD parameters +} whenever there exists a Hadamard matrix of size $u$.
To our knowledge, the only previously known QSDs with parameters~\eqref{equation.special QSD parameters +} had either $u=2^e$ for some $e\geq0$ or $u=6$~\cite{BrackenMW06}.
As such, this result provides a new infinite family of such designs.

To be clear, Theorem~\ref{theorem.QSD} can be applied to any real flat ETF,
and these do not necessarily have parameters of the form $(u(2u-1),4u^2)$ or $(u(2u+1),4u^2)$.
In particular, if there exists an $\RBIBD(\hat{v},\hat{k},1,\hat{r},\hat{b})$ and Hadamard matrices of size $\hat{r}+1$ and \smash{$\frac{\hat{v}}{\hat{k}}$}, then~\cite{JasperMF14} gives a real flat Kirkman ETF with parameters $(d,n)=(\hat{b},\hat{v}(\hat{r}+1))$.
When $\hat{k}>2$, such an ETF is neither of the types characterized in Corollary~\ref{corollary.existence of real flat}.
Applying Theorem~\ref{theorem.QSD} to it and simplifying the expressions for the parameters of the resulting QSD gives the following result:

\begin{corollary}
\label{corollary.RBIBD and Hadamard imply QSD}
If there exists an $\RBIBD(\hat{v},\hat{k},1,\hat{r},\hat{b})$ and Hadamard matrices of size $\hat{r}+1$ and \smash{$\frac{\hat{v}}{\hat{k}}$},
then there exists a QSD with parameters
\begin{equation*}
(v,k,\lambda,r,b,x,y)
=(\hat{b},\tfrac{\hat{v}(\hat{r}-1)}{2\hat{k}},\tfrac{\hat{v}(\hat{r}-1)-2\hat{k}}{4},\tfrac{(\hat{r}-1)(\hat{v}+\hat{k}-1)}{2},\hat{r}(\hat{v}+\hat{k}-1),\tfrac{\hat{v}(\hat{r}-3)}{4\hat{k}},\tfrac{\hat{v}(\hat{r}-1)}{4\hat{k}}),
\quad
w=\tfrac{\hat{v}}{\hat{k}}.
\end{equation*}
\end{corollary}

Moreover, using a class of RBIBDs that were overlooked in~\cite{JasperMF14},
one can show that instances of such ETFs with $\hat{k}>2$ exist.
To be precise, for any positive integer $i$ and any prime power $q$,
the finite projective space $\operatorname{PG}(2^{i+1}-1,q)$ is resolvable~\cite{Beutelspacher74}
and so is an RBIBD with parameters
\begin{equation*}
(\hat{v},\hat{k},\hat{\lambda},\hat{r},\hat{b})
=(\tfrac{q^{2^{i+1}}-1}{q-1},\,q+1,\,1,\,\tfrac{q^{2^{i+1}-1}-1}{q-1},\,\tfrac{(q^2)^{2^i}-1}{q^2-1}\tfrac{q^{2^{i+1}-1}-1}{q-1}).
\end{equation*}
In order for Hadamard matrices of size \smash{$\frac{\hat{v}}{\hat{k}}$} and $\hat{r}+1$ to exist,
these quantities are necessarily divisible by $4$, which happens when $i\geq 2$ and $q\equiv 1\bmod 4$.
In particular, taking $i=2$ and $q=5$ gives an $\RBIBD(97656,6,1,19531,317886556)$;
since Hadamard matrices of size \smash{$\frac{\hat{v}}{\hat{k}}=16276$} and $r+1=19532$ exist~\cite{CraigenK07},
this implies the existence of a real flat Kirkman ETF with parameters $(317886556,1907416992)$.
Though these parameters are admittedly large, this to our knowledge is the first example of a real flat ETF whose parameters are not of the form $(u(2u-1),4u^2)$ or $(u(2u+1),4u^2)$.
Applying Theorem~\ref{theorem.QSD} to this ETF \`{a} la Corollary~\ref{corollary.RBIBD and Hadamard imply QSD} gives a new QSD with parameters
$(317886556,158943278,476829831,953659665,1907416991,79459432,79467570),$
which is not of the form~\eqref{equation.special QSD parameters -} or~\eqref{equation.special QSD parameters +}.
Applying Theorem~A of~\cite{McGuire97} to this QSD or alternatively Theorem~3 of~\cite{JasperMF14} to this ETF gives a new code that achieves equality in the Grey-Rankin bound.

It is also noteworthy that Theorem~\ref{theorem.QSD} sometimes yields non-flat real ETFs with the same $(d,n)$ parameters as flat ones.
For example, one can show that QSD parameters $(v,k,\lambda,r,b,x,y)$ with $\lambda=1$ satisfy~\eqref{equation.QSD parameters} if and only if $v=2k^2-k$.
An infinite number of such QSDs exist,
being instances of known $\BIBD(2k^2-k,k,1,2k+1,4k^2-1)$.
This includes whenever $k=2^e$ for some $e\geq 1$---a type of \textit{Denniston design}---and also whenever $k=3,5,6,7$~\cite{MathonR07}.
For any such design with $k>2$, applying Theorem~\ref{theorem.QSD} to it yields a non-flat ETF with parameters $(d,n)=(2k^2-k,4k^2)$ and entries valued either $1$, $\delta=\frac1{2k-1}(1\pm\sqrt{2k})$ or $\delta+\varepsilon=\frac1{2k-1}[1\mp(2k-2)\sqrt{2k}]$.
At the same time, taking $u=k$ in Corollary~\ref{corollary.existence of real flat} gives a real flat ETF that also has $(d,n)=(2k^2-k,4k^2)$ whenever there exists a Hadamard matrix of size $k$ or when $k=6$.
In particular, there are an infinite number of instances in which two different QSDs give two different ETFs---one flat and the other not---with the same $(d,n)$ parameters.
We also see here that when $k=3,5,7$, Theorem~\ref{theorem.QSD} produces ETFs with $(d,n)$ being $(15,36)$, $(45,100)$ and $(91,196)$, respectively;
since these values of $d$ are odd, these ETFs cannot be flat.

\subsection{Necessary integrality conditions for Hadamard ETFs}

We conclude this section by strengthening the necessary conditions on the existence of real flat ETFs given in Theorem~\ref{theorem.QSD} by combining them with other known necessary conditions on real and unital ETFs in the literature, specifically those given in~\cite{SustikTDH07}:
\begin{corollary}
If $n-1>d>1$ and there exists a real flat ETF with parameters $(d,n)$, then
\begin{equation}
\label{eqref.integrality conditions}
\bigbracket{\tfrac{d(n-1)}{n-d}}^{\frac12},
\quad
\bigbracket{\tfrac{(n-d)(n-1)}{n-d}}^{\frac12},
\quad
\bigbracket{\tfrac{d(n-d)}{n-1}}^{\frac12},
\end{equation}
are integers, and are necessarily odd, odd and even, respectively.
Moreover, $n$ is divisible by $16$.
\end{corollary}

\begin{proof}
Under these assumptions we have $n\neq 2d$: otherwise Theorem~\ref{theorem.QSD} gives that
\begin{equation*}
4w^2
=\tfrac{4d(n-d)}{n-1}
=\tfrac{4d^2}{2d-1}
=2d+1+\tfrac1{2d-1}
\end{equation*}
is an integer divisible by $16$, implying $d=1$.
Since $n-1>d>1$ where $n\neq 2d$,
Theorem~A of~\cite{SustikTDH07} gives that the first two quantities in~\eqref{eqref.integrality conditions} are indeed odd integers.
Moreover, Theorem~B of~\cite{SustikTDH07} gives that the third quantity in~\eqref{eqref.integrality conditions} is an integer.
In fact, in the notation of Theorem~\ref{theorem.QSD}, this third quantity is the parameter $w$, and so is necessarily even.
(Thus Theorem~\ref{theorem.QSD} shows, for example, that real ETFs with parameters $(d,n)=(15,36)$ cannot be flat, despite satisfying all necessary conditions on such ETFs given in~\cite{SustikTDH07}.)
For the final conclusion,
note that being an odd square, \smash{$\frac{d(n-1)}{n-d}$} is congruent to either $1$ or $9$ modulo $16$.
In the first case, we thus have $d(n-1)\equiv n-d\bmod 16$ and so $(d-1)n\equiv0\bmod 16$;
since $d$ is even, $d-1$ is a unit in $\bbZ_{16}$, implying $n\equiv 0\bmod 16$.
Similarly, in the second case we have $d(n-1)\equiv 9(n-d)\bmod 16$ and so $(d-9)n\equiv8d \bmod 16$;
since $d$ is even, this again implies $n\equiv0\bmod16$.
\end{proof}

\section{Miscellanea}

In this section, we present two other results regarding Hadamard ETFs.
The first of these results show how, in a special case, we can take tensor products of two Hadamard ETFs to produce another.
The second result generalizes the well-known Gerzon bound to provide new necessary conditions on such ETFs.

In Definition~\ref{definition.flat}, we define a Hadamard ETF to be an ETF whose synthesis operator is a submatrix of a Hadamard matrix.
One may also consider ETFs whose Gram matrices are related to Hadamard matrices.
In particular, there is a well-known equivalence between symmetric Hadamard ETFs with constant diagonal and real ETFs whose parameters $(d,n)$ satisfy $d=\frac12(n-\sqrt{n})$~\cite{HolmesP04};
see~\cite{BrouwerH12} for a review of the literature of such Hadamard matrices.
This idea has also been generalized to the complex setting~\cite{BodmannPT09,Szollosi13}.

Here, the idea is that if \smash{$\set{\bfphi_j}_{j=1}^{n}$} is an ETF for $\bbF^d$ with $d=\frac12(n-\sqrt{n})$ and if we assume without loss of generality that \smash{$\norm{\bfphi_j}^2=\frac{2d}{\sqrt{n}}$} for all $j$,
then $\bfH=\sqrt{n}\,\bfI-\bfPhi^*\bfPhi$ is a self-adjoint possibly-complex Hadamard matrix whose diagonal entries are one.
Indeed, the diagonal entries of $\bfH$ are \smash{$\sqrt{n}-\frac{2d}{\sqrt{n}}=1$},
and the off-diagonal entries have modulus \smash{$\frac{2d}{\sqrt{n}}[\frac{n-d}{d(n-1)}]^{\frac12}=1$}.
At the same time, the fact that \smash{$\bfPhi\bfPhi^*=\frac{n}{d}\frac{2d}{\sqrt{n}}\bfI=2\sqrt{n}\,\bfI$} implies that the eigenvalues of $\bfH$ are $\pm\sqrt{n}$, implying $\bfH\bfH^*=n\bfI$.
Conversely, if $\bfH$ is any such Hadamard matrix,
then $\bfG=\sqrt{n}\,\bfI-\bfH$ is positive-semidefinite having eigenvalues $0$ and $2\sqrt{n}$ and trace $n(\sqrt{n}-1)$.
Diagonalizing $\bfG$ thus reveals it to be the Gram matrix of an ETF for $\bbF^d$ where $d$ is the multiplicity of $2\sqrt{n}$, namely $d=\frac12(n-\sqrt{n})$.

Given any two self-adjoint possibly-complex Hadamard matrices whose diagonal entries are one,
we can take their tensor product to construct another such matrix.
As noted in~\cite{BodmannPT09,Szollosi13,GoyenecheT16},
this fact along with the above equivalence implies that if there exists ETFs \smash{$\set{\bfphi_j}_{j=1}^{n_1}$} and \smash{$\set{\bfpsi_j}_{j=1}^{n_2}$} for $\bbF^{d_1}$ and $\bbF^{d_2}$, respectively, where $d_1=\frac12(n_1-\sqrt{n_1})$ and $d_2=\frac12(n_2-\sqrt{n_2})$,
then there exists an ETF consisting of $n_1n_2$ vectors for $\bbF^{d_3}$ where $d_3=\frac12(n_1n_2-\sqrt{n_1n_2})$.
We now present an alternative proof of this fact that constructs the synthesis operator of the resulting ETF explicitly.
In the special case where \smash{$\set{\bfphi_j}_{j=1}^{n_1}$} and \smash{$\set{\bfpsi_j}_{j=1}^{n_2}$} are Hadamard ETFs,
this construction implies that the $n_1n_2$-vector ETF is as well.

\begin{theorem}
\label{theorem.tensor product ETF}
Let \smash{$\set{\bfphi_j}_{j=1}^{n_1}$} and \smash{$\set{\bfpsi_j}_{j=1}^{n_2}$} be ETFs for $\bbF^{d_1}$ and $\bbF^{d_2}$, respectively, where the ETF parameters satisfy $d_1=\frac12(n_1-\sqrt{n_1})$ and $d_2=\frac12(n_2-\sqrt{n_2})$.
Let \smash{$\set{\widetilde{\bfphi}_j}_{j=1}^{n_1}$} and \smash{$\set{\widetilde{\bfpsi}_j}_{j=1}^{n_2}$} be any Naimark complements of \smash{$\set{\bfphi_j}_{j=1}^{n_1}$} and \smash{$\set{\bfpsi_j}_{j=1}^{n_2}$} in $\bbF^{n_1-d_1}$ and $\bbF^{n_2-d_2}$, respectively.
Then
\begin{equation}
\label{equation.tensor product ETF}
\set{(\bfphi_j\otimes\widetilde{\bfpsi}_{j'})\oplus(\widetilde{\bfphi}_j\otimes\bfpsi_{j'})}_{j=1,}^{n_1}\,_{j'=1}^{n_2},
\quad
\set{(\bfphi_j\otimes\bfpsi_{j'})\oplus(\widetilde{\bfphi}_j\otimes\widetilde{\bfpsi}_{j'})}_{j=1,}^{n_1}\,_{j'=1}^{n_2},
\end{equation}
are Naimark complementary ETFs for $\bbF^{d_3}$ and $\bbF^{n_1n_2-d_3}$, respectively, where $d_3=\frac12(n_1n_2-\sqrt{n_1n_2})$.
In particular, if \smash{$\set{\bfphi_j}_{j=1}^{n_1}$} and \smash{$\set{\bfpsi_j}_{j=1}^{n_2}$} are Hadamard, then so are the ETFs in~\eqref{equation.tensor product ETF}.
\end{theorem}

\begin{proof}
Without loss of generality, \smash{$\norm{\bfphi_j}^2=d_1$}, \smash{$\norm{\bfpsi_j}^2=d_2$}, \smash{$\norm{\widetilde{\bfphi}_j}^2=n_1-d_1$},
and \smash{$\norm{\widetilde{\bfpsi}_j}^2=n_2-d_2$} for all $j$.
Letting $\bfPhi$, $\bfPsi$, $\widetilde{\bfPhi}$ and $\widetilde{\bfPsi}$
denote the synthesis operators for \smash{$\set{\bfphi_j}_{j=1}^{n_1}$}, \smash{$\set{\bfpsi_j}_{j=1}^{n_2}$} and their Naimark complements, respectively,
we have that
\begin{equation}
\label{equation.proof of tensor product ETF 1}
\bfPhi\bfPhi^*=n_1\bfI,
\quad
\widetilde{\bfPhi}\widetilde{\bfPhi}^*=n_1\bfI,
\quad
\bfPhi\widetilde{\bfPhi}^*=\bfzero,
\quad
\bfPsi\bfPsi^*=n_2\bfI,
\quad
\widetilde{\bfPsi}\widetilde{\bfPsi}^*=n_2\bfI,
\quad
\bfPsi\widetilde{\bfPsi}^*=\bfzero.
\end{equation}
The synthesis operators of the vector sequences in~\eqref{equation.tensor product ETF} are the
$d_3\times n_1n_2$ and $d_4\times n_1n_2$ matrices
\begin{equation*}
\left[\begin{array}{c}
\bfPhi\otimes\widetilde{\bfPsi}\\
\widetilde{\bfPhi}\otimes\bfPsi
\end{array}\right],
\quad
\left[\begin{array}{c}
\bfPhi\otimes\bfPsi\\
\widetilde{\bfPhi}\otimes\widetilde{\bfPsi}
\end{array}\right],
\end{equation*}
respectively,
where the fact that $d_1=\frac12(n_1-\sqrt{n_1})$ and $d_2=\frac12(n_2-\sqrt{n_2})$ implies
\begin{align*}
d_3
&=d_1(n_2-d_2)+(n_1-d_1)d_2
=\tfrac12(n_1n_2-\sqrt{n_1n_2}),\\
d_4
&=d_1d_2+(n_1-d_1)(n_2-d_2)
=\tfrac12(n_1n_2+\sqrt{n_1n_2}).
\end{align*}
Here,~\eqref{equation.proof of tensor product ETF 1} implies
\begin{equation*}
\setlength{\arraycolsep}{3pt}
\left[\begin{array}{c}
\bfPhi\otimes\widetilde{\bfPsi}\\
\widetilde{\bfPhi}\otimes\bfPsi\\
\bfPhi\otimes\bfPsi\\
\widetilde{\bfPhi}\otimes\widetilde{\bfPsi}
\end{array}\right]
\left[\begin{array}{c}
\bfPhi\otimes\widetilde{\bfPsi}\\
\widetilde{\bfPhi}\otimes\bfPsi\\
\bfPhi\otimes\bfPsi\\
\widetilde{\bfPhi}\otimes\widetilde{\bfPsi}
\end{array}\right]^*
=\left[\begin{array}{cccc}
\bfPhi\bfPhi^*\otimes\widetilde{\bfPsi}\widetilde{\bfPsi}^*&\bfPhi\widetilde{\bfPhi}^*\otimes\widetilde{\bfPsi}\bfPsi^*&\bfPhi\bfPhi^*\otimes\widetilde{\bfPsi}\bfPsi^*&\bfPhi\widetilde{\bfPhi}^*\otimes\widetilde{\bfPsi}\widetilde{\bfPsi}^*\\
\widetilde{\bfPhi}\bfPhi^*\otimes\bfPsi\widetilde{\bfPsi}^*&\widetilde{\bfPhi}\widetilde{\bfPhi}^*\otimes\bfPsi\bfPsi^*&\widetilde{\bfPhi}\bfPhi^*\otimes\bfPsi\bfPsi^*&\widetilde{\bfPhi}\widetilde{\bfPhi}^*\otimes\bfPsi\widetilde{\bfPsi}^*\\
\bfPhi\bfPhi^*\otimes\bfPsi\widetilde{\bfPsi}^*&\bfPhi\widetilde{\bfPhi}^*\otimes\bfPsi\bfPsi^*&\bfPhi\bfPhi^*\otimes\bfPsi\bfPsi^*&\bfPhi\widetilde{\bfPhi}^*\otimes\bfPsi\widetilde{\bfPsi}^*\\
\widetilde{\bfPhi}\bfPhi^*\otimes\widetilde{\bfPsi}\widetilde{\bfPsi}^*&\widetilde{\bfPhi}\widetilde{\bfPhi}^*\otimes\widetilde{\bfPsi}\bfPsi^*&\widetilde{\bfPhi}\bfPhi^*\otimes\widetilde{\bfPsi}\bfPsi^*&\widetilde{\bfPhi}\widetilde{\bfPhi}^*\otimes\widetilde{\bfPsi}\widetilde{\bfPsi}^*\\
\end{array}\right]
=n_1n_2\bfI.
\end{equation*}
Since $d_3+d_4=n_1n_2$, this shows that~\eqref{equation.tensor product ETF} indeed defines Naimark complementary tight frames.
What remains is to show that one of these two sequences of vectors is equiangular.
For the second sequence in particular, Naimark complementarity implies
\begin{align*}
&\ip{(\bfphi_j\otimes\bfpsi_{j'})\oplus(\widetilde{\bfphi}_j\otimes\widetilde{\bfpsi}_{j'})}{(\bfphi_{j''}\otimes\bfpsi_{j'''})\oplus(\widetilde{\bfphi}_{j''}\otimes\widetilde{\bfpsi}_{j'''})}\\
&\quad=\ip{\bfphi_j}{\bfphi_{j''}}\ip{\bfpsi_{j'}}{\bfpsi_{j'''}}
+\ip{\widetilde{\bfphi}_j}{\widetilde{\bfphi}_{j''}}\ip{\widetilde{\bfpsi}_{j'}}{\widetilde{\bfpsi}_{j'''}}\\
&\quad=\ip{\bfphi_j}{\bfphi_{j''}}\ip{\bfpsi_{j'}}{\bfpsi_{j'''}}
+\left\{\begin{array}{rl}
n_1-\ip{\bfphi_j}{\bfphi_{j''}},& j=j''\\
   -\ip{\bfphi_j}{\bfphi_{j''}},& j\neq j''
\end{array}\right\}
\left\{\begin{array}{rl}
n_2-\ip{\bfpsi_{j'}}{\bfpsi_{j'''}},&  j'=j'''\\
   -\ip{\bfpsi_{j'}}{\bfpsi_{j'''}},&j'\neq j'''
\end{array}\right\}\\
&\quad=\left\{\begin{array}{cl}
d_1d_2+(n_1-d_1)(n_2-d_2),& \ j=j'',j'=j''',\smallskip\\
-(n_1-2d_1)\ip{\bfpsi_{j'}}{\bfpsi_{j'''}},&\ j=j'', j'\neq j''',\smallskip\\
-\ip{\bfphi_j}{\bfphi_{j''}}(n_2-2d_2),&\ j\neq j'', j'=j''',\smallskip\\
2\ip{\bfphi_j}{\bfphi_{j''}}\ip{\bfpsi_{j'}}{\bfpsi_{j'''}},&\ j\neq j'', j'\neq j'''.
\end{array}\right.
\end{align*}
As such, this sequence is equiangular if
$\abs{\ip{\bfphi_j}{\bfphi_{j''}}}=\frac12(n_1-2d_1)=\frac12\sqrt{n_1}$ for all $j\neq j''$ and $\abs{\ip{\bfpsi_{j'}}{\bfpsi_{j'''}}}=\frac12(n_2-2d_2)=\frac12\sqrt{n_2}$ for all $j'\neq j'''$;
these hold since \smash{$\set{\bfphi_j}_{j=1}^{n_1}$}, \smash{$\set{\bfpsi_j}_{j=1}^{n_2}$} are equiangular and achieve the Welch bound with $d_1=\frac12(n_1-\sqrt{n_1})$ and $d_2=\frac12(n_2-\sqrt{n_2})$.
\end{proof}

As an example of the previous result, note that for $n_1=n_2=4$ and $d_1=d_2=\frac12(4-\sqrt{4})=1$,
we can take
$\bfPhi=\bfPsi=\left[\begin{array}{rrrr}1&1&1&1\end{array}\right]$
and take \smash{$\widetilde{\bfPhi}=\widetilde{\bfPsi}$} to be the $3\times 4$ matrix~\eqref{equation.3x4 real flat regular simplex} formed by the remaining three rows of the canonical $4\times 4$ Hadamard matrix.
Applying Theorem~\ref{theorem.tensor product ETF} then produces a Hadamard ETF with $(d,n)=(6,16)$:
\begin{align*}
\left[\begin{array}{c}
\bfPhi\otimes\widetilde{\bfPhi}\\
\widetilde{\bfPhi}\otimes\bfPhi
\end{array}\right]
=\left[\begin{array}{rrrrrrrrrrrrrrrr}
 1&-1& 1&-1& 1&-1& 1&-1& 1&-1& 1&-1& 1&-1& 1&-1\\
 1& 1&-1&-1& 1& 1&-1&-1& 1& 1&-1&-1& 1& 1&-1&-1\\
 1&-1&-1& 1& 1&-1&-1& 1& 1&-1&-1& 1& 1&-1&-1& 1\\
 1& 1& 1& 1&-1&-1&-1&-1& 1& 1& 1& 1&-1&-1&-1&-1\\
 1& 1& 1& 1& 1& 1& 1& 1&-1&-1&-1&-1&-1&-1&-1&-1\\
 1& 1& 1& 1&-1&-1&-1&-1&-1&-1&-1&-1& 1& 1& 1& 1
\end{array}\right].
\end{align*}
Repeatedly applying Theorem~\ref{theorem.tensor product ETF} to this and other resulting ETFs yields real flat ETFs with $d=\frac12(n-\sqrt{n})$ where $n=2^{2(e+1)}$ for any $e\geq 0$.
Though Hadamard ETFs of this size are already well-known,
this method of construction is trivial, making no use of the theory of difference sets or BIBDs.
In light of Corollary~\ref{corollary.existence of real flat}, one may also be tempted to apply Theorem~\ref{theorem.tensor product ETF} to real flat ETFs arising from MOLS.
This is challenging, since it seems the ETFs arising from the QSDs of~\cite{BrackenMW06} with parameters~\eqref{equation.special QSD parameters -} and~\eqref{equation.special QSD parameters +} are not Naimark complements by default.
Moreover, this has little payoff: at best, such a combination of Theorem~\ref{theorem.tensor product ETF} and Corollary~\ref{corollary.existence of real flat} gives a way to combine two QSDs with parameters~\eqref{equation.special QSD parameters -} for some even integers $u_1$, $u_2$ so as to produce another such QSD with $u_3=2u_1u_2$;
since $8$ divides $u_3$, such QSDs are probably more easily obtained via a $u_3\times u_3$ Hadamard matrix.

\subsection{The Gerzon bound for Hadamard ETFs}

It has long been known~\cite{LemmensS73} that if there exists $n$ equiangular noncollinear lines in $\bbR^d$ then we necessarily have $n\leq\binom{d+1}{2}$; in the complex case we necessarily have $n\leq d^2$.
An alternative proof of these facts is discussed in~\cite{Tropp05}.
In~\cite{HolmesP04} it is noted that in order for an ETF to exist its Naimark complement must also satisfy these restrictions, that is, we also need \smash{$n\leq\binom{n-d+1}{2}$} and $n\leq(n-d)^2$ in the real and complex cases, respectively.
Moreover, in the case of complex unital ETFs, this upper bound can be strengthened so as to require $n\leq d^2-d+1$~\cite{SustikTDH07}.
We now refine these ideas to obtain necessary conditions on the existence of possibly-complex Hadamard ETFs.

Given any $\set{\bfphi_j}_{j=1}^{n}$ in $\bbF^d$,
consider their outer products $\set{\bfphi_j^{}\bfphi_j^*}_{j=1}^{n}$ which lie in the real Hilbert space of all self-adjoint matrices in $\bbF^{d\times d}$.
The Frobenius inner product of any two such outer products is
$\ip{\bfphi_{j}^{}\bfphi_{j}^{*}}{\bfphi_{j'}^{}\bfphi_{j'}^{*}}_\Fro
=\Tr(\bfphi_{j}^{}\bfphi_{j}^{*}\bfphi_{j'}^{}\bfphi_{j'}^{*})
=\Tr(\bfphi_{j}^{*}\bfphi_{j'}^{}\bfphi_{j'}^{*}\bfphi_{j}^{})
=\abs{\ip{\bfphi_j}{\bfphi_{j'}}}^2$.
In particular, if $\set{\bfphi_j}_{j=1}^{n}$ is any sequence of noncollinear equiangular vectors, then the $n\times n$ Gram matrix of $\set{\bfphi_j^{}\bfphi_j^*}_{j=1}^{n}$ is $(\beta^2-\gamma^2)\bfI+\gamma^2\bfJ$ for some scalars $0\leq\gamma^2<\beta^2$.
Such a Gram matrix has rank $n$,
implying \smash{$\set{\bfphi_j^{}\bfphi_j^*}_{j=1}^{n}$} is linearly independent.
This implies that $n$ is at most the dimension of $\set{\bfB\in\bbF^{d\times d}: \bfB^*=\bfB}$, which is $\binom{d+1}{2}$ or $d^2$ depending on whether $\bbF=\bbR$ or $\bbF=\bbC$, respectively.

Now assume that $\set{\bfphi_j}_{j=1}^{n}$ is a flat ETF for $\bbF^d$ where $1<d<n-1$.
Here, $\set{\bfphi_j^{}\bfphi_j^*}_{j=1}^{n}$ is a sequence of linearly independent vectors in the subspace of \smash{$\set{\bfB\in\bbF^{d\times d}: \bfB^*=\bfB}$} that consists of all such matrices with constant diagonals.
Computing the dimensions of this subspace, we thus have $n\leq\tfrac12d^2-\tfrac12d+1$ or $n\leq d^2-d+1$ depending on whether $\bbF=\bbR$ or $\bbF=\bbC$, respectively.
If we further assume that $\set{\bfphi_j}_{j=1}^{n}$ is possibly-complex Hadamard,
then letting $\set{\widetilde{\bfphi}_j}_{j=1}^{n}$ in $\bbF^{n-d}$ be a flat Naimark complement for it,
we necessarily have $n\leq\tfrac12(n-d)^2-\tfrac12(n-d)+1$ or $n\leq(n-d)^2-(n-d)+1$ when $\bbF=\bbR$ or $\bbF=\bbC$, respectively.
Solving for $n$ in these inequalities gives the following result:

\begin{theorem}
Let $\set{\bfphi_j}_{j=1}^{n}$ be an ETF for $\bbF^d$ where $1<d<n-1$.
\begin{enumerate}
\renewcommand{\labelenumi}{(\alph{enumi})}
\item
If $\bbF=\bbC$ and $\set{\bfphi_j}_{j=1}^{n}$ is flat then $n\leq d^2-d+1$.
\item
If $\bbF=\bbC$ and $\set{\bfphi_j}_{j=1}^{n}$ is complex Hadamard then
$d+d^{\frac12}+1\leq n\leq d^2-d+1$.
\item
If $\bbF=\bbR$ and $\set{\bfphi_j}_{j=1}^{n}$ is flat then $n\leq\tfrac12d^2-\tfrac12d+1$;
\item
If $\bbF=\bbR$ and $\set{\bfphi_j}_{j=1}^{n}$ is Hadamard then
$d+(2d+\tfrac14)^{\frac12}+\tfrac32\leq n\leq \tfrac12d^2-\tfrac12d+1$.
\end{enumerate}
\end{theorem}

We remark that the above bounds are achieved infinitely often in the complex setting:
for any prime power $q$, there exists a $d\times n$ harmonic ETF arising from a Singer difference set that has $d=q+1$ and $n=q^2+q+1$~\cite{StrohmerH03,XiaZG05}, meaning $n=d^2-d+1$; its $q^2\times (q^2+q+1)$ flat Naimark complement achieves equality in the lower bound.

\section*{Acknowledgments}
This work was partially supported by NSF DMS 1321779,
AFOSR F4FGA06060J007 and AFOSR Young Investigator Research Program award F4FGA06088J001.
The views expressed in this article are those of the authors and do not reflect the official policy or position of the United States Air Force, Department of Defense, or the U.S.~Government.

\end{document}